\documentclass[12pt,english,a4paper]{smfart}

\usepackage[T1]{fontenc}
\usepackage{lmodern}
\usepackage{smfthm}
\usepackage[headings]{fullpage}
\usepackage{amssymb,amscd}

\newcommand{\Qp}{\mathbf{Q}_p}
\newcommand{\Cp}{\mathbf{C}_p}
\newcommand{\Zp}{\mathbf{Z}_p}
\newcommand{\FF}{\mathbf{F}}
\newcommand{\NN}{\mathbf{N}}
\newcommand{\ZZ}{\mathbf{Z}}
\newcommand{\OO}{\mathcal{O}}
\newcommand{\Qpbar}{\overline{\mathbf{Q}}_p}

\renewcommand{\phi}{\varphi}

\renewcommand{\projlim}{\varprojlim}
\renewcommand{\geq}{\geqslant}
\renewcommand{\leq}{\leqslant} 
\newcommand{\calR}{\mathcal{R}}

\renewcommand{\O}{\mathrm{O}}

\newcommand{\Gal}{\mathrm{Gal}}
\newcommand{\Mat}{\mathrm{Mat}}
\newcommand{\End}{\mathrm{End}}
\newcommand{\Hom}{\mathrm{Hom}}
\newcommand{\M}{\mathrm{M}}
\newcommand{\GL}{\mathrm{GL}}
\newcommand{\Id}{\mathrm{Id}}
\newcommand{\Nm}{\mathrm{N}}
\newcommand{\calC}{\mathcal{C}}
\newcommand{\Lie}{\mathrm{Lie}}
\newcommand{\Sol}{\mathrm{Sol}}

\newcommand{\vp}{\mathrm{val}_p}
\newcommand{\ve}{\mathrm{val}_{\mathbf{E}}}
\newcommand{\an}{\mathrm{an}}
\newcommand{\la}{\mathrm{la}}
\newcommand{\pa}{\mathrm{pa}}
\newcommand{\fpa}{\text{$F\!\mbox{-}\mathrm{pa}$}}
\newcommand{\dan}{\text{$\mbox{-}\mathrm{an}$}}
\newcommand{\fan}{\text{$F\!\mbox{-}\mathrm{la}$}}
\newcommand{\fla}{\text{$F\!\mbox{-}\mathrm{la}$}}
\newcommand{\unr}{\mathrm{unr}}
\newcommand{\cyc}{\mathrm{cyc}}
\newcommand{\bfont}{\mathbf{B}}
\newcommand{\afont}{\mathbf{A}}
\newcommand{\efont}{\mathbf{E}}
\newcommand{\bt}{\widetilde{\mathbf{B}}}
\newcommand{\at}{\tilde{\mathbf{A}}}
\newcommand{\et}{\widetilde{\mathbf{E}}}
\newcommand{\atplus}{\tilde{\mathbf{A}}^+}
\newcommand{\btplus}{\widetilde{\mathbf{B}}^+}
\newcommand{\etplus}{\widetilde{\mathbf{E}}^+}
\newcommand{\bdr}{\mathbf{B}_{\mathrm{dR}}}  
\newcommand{\btrig}[2]{\widetilde{\mathbf{B}}^{\dagger #1}_{\mathrm{rig} #2}} 
\newcommand{\brig}[2]{\mathbf{B}^{\dagger #1}_{\mathrm{rig} #2}}

\newcommand{\bdag}[1]{\mathbf{B}^{\dagger #1}}
\newcommand{\dtilde}{\widetilde{\mathrm{D}}}
\renewcommand{\ddag}[1]{\mathrm{D}^{\dagger #1}}
\newcommand{\dtdag}[1]{\widetilde{\mathrm{D}}^{\dagger #1}}
\newcommand{\drig}[1]{\mathrm{D}^{\dagger #1}_{\mathrm{rig}}}
\newcommand{\dtrig}[1]{\widetilde{\mathrm{D}}^{\dagger #1}_{\mathrm{rig}}}
\newcommand{\dfont}{\mathrm{D}}
\newcommand{\dsen}{\mathrm{D}_{\mathrm{Sen}}}
\newcommand{\mfont}{\mathrm{M}}
\newcommand{\dpar}[1]{(\!( #1 )\!)}
\newcommand{\dcroc}[1]{[\![ #1 ]\!]}
\newcommand{\LT}{\operatorname{LT}}
\newcommand{\chicyc}{\chi_{\cyc}}
\newcommand{\chilt}{\chi_\pi}
\newcommand{\rig}{\mathrm{rig}}
\newcommand{\emb}{\Sigma}
\newcommand{\LA}{\mathrm{LA}}
\newcommand{\weil}{\mathrm{W}}
\newcommand{\ubar}{\overline{u}}
\newcommand{\Qpnr}{\mathbf{Q}_p^{\unr}}
\newcommand{\cbf}{\mathbf{c}}
\newcommand{\ibf}{\mathbf{i}}
\newcommand{\jbf}{\mathbf{j}}
\newcommand{\kbf}{\mathbf{k}}
\newcommand{\ybf}{\mathbf{y}}
\newcommand{\unbf}{\mathbf{1}}

\author{Laurent Berger}

\title[Multivariable $(\varphi,\Gamma)$-modules and locally analytic vectors]
{Multivariable $(\varphi,\Gamma)$-modules and \\ locally analytic vectors}

\subjclass{11F; 11S; 12H; 13J; 22E; 46S}

\keywords{$(\varphi,\Gamma)$-module; locally analytic vector; $p$-adic period; Lubin-Tate group; $p$-adic monodromy}

\thanks{The author's work was supported in part by the Agence Nationale de la Recherche (ANR) grant Th\'eHopaD (Th\'eorie de Hodge $p$-adique et D\'evelop\-pements) ANR-11-BS01-005}

\begin{document}

\begin{abstract}
Let $K$ be a finite extension of $\mathbf{Q}_p$ and let $G_K = \mathrm{Gal}(\overline{\mathbf{Q}}_p/K)$. There is a very useful classification of $p$-adic representations of $G_K$ in terms of cyclotomic $(\varphi,\Gamma)$-modules (\emph{cyclotomic} means that $\Gamma=\mathrm{Gal}(K_\infty/K)$ where $K_\infty$ is the cyclotomic extension of $K$). One particularly convenient feature of the cyclotomic theory is the fact that the $(\varphi,\Gamma)$-module attached to any $p$-adic representation is overconvergent.

Questions pertaining to the $p$-adic local Langlands correspondence lead us to ask for a  generalization of the theory of $(\varphi,\Gamma)$-modules, with the cyclotomic extension replaced by an infinitely ramified $p$-adic Lie extension $K_\infty / K$. It is not clear what shape such a generalization should have in general. Even in the case where we have such a generalization, namely the case of a Lubin-Tate extension, most $(\varphi,\Gamma)$-modules fail to be overconvergent.

In this article, we develop an approach that gives a solution to both problems at the same time, by considering the locally analytic vectors for the action of $\Gamma$ inside some big modules defined using  Fontaine's rings of periods.  

We show that, in the cyclotomic case, we recover the usual overconvergent $(\varphi,\Gamma)$-modules. In the Lubin-Tate case, we can prove, as an application of our theory, a folklore conjecture in the field stating that $(\varphi,\Gamma)$-modules attached to $F$-analytic representations are overconvergent. 
\end{abstract}

\maketitle

\tableofcontents

\setlength{\baselineskip}{18pt}

\section*{Introduction}
\label{nota}

Let $K$ be a finite extension of $\Qp$ and let $G_K=\Gal(\Qpbar/K)$. The basic idea of $p$-adic Hodge theory is to construct an intermediate extension $K \subset K_\infty \subset \Qpbar$ such that $K_\infty/K$ is simple enough, but still contains most of the ramification of $\Qpbar/K$ (we say that $K_\infty/K$ is deeply ramified, see \cite{CG96}). This is for example the case if $K_\infty / K$ is an infinitely ramified $p$-adic Lie extension (\cite{S72} and \cite{CG96}). The usual choice for $K_\infty/K$ is the cyclotomic extension. 

One important application of this idea is Fontaine's construction \cite{F90} of cyclotomic $(\phi,\Gamma)$-modules. By a theorem of Cherbonnier and Colmez \cite{CC98}, these cyclotomic $(\phi,\Gamma)$-modules are always overconvergent; this is a fundamental result which allows us to relate Fontaine's $(\phi,\Gamma)$-modules and classical $p$-adic Hodge theory. The resulting $(\phi,\Gamma)$-modules give rise to free modules of finite rank over the Robba ring. 

If $V$ is a $p$-adic representation of $G_K$, then the cyclotomic $(\phi,\Gamma)$-module $\drig{}(V)$ over the Robba ring attached to $V$ can be constructed in the following way. Let $K_\infty$ be the cyclotomic extension of $K$, let $H_K=\Gal(\Qpbar/K_\infty)$ and let $\Gamma_K=\Gal(K_\infty/K)$. Let $\btrig{}{}$ be one of the big rings of $p$-adic periods \cite{LB2}, let $\btrig{}{,K} = (\btrig{}{})^{H_K}$ and let $\dtrig{}{}(V) = (\btrig{}{} \otimes_{\Qp} V)^{H_K}$. By \'etale descent, we have $\btrig{}{} \otimes_{\btrig{}{,K}} \dtrig{}{}(V) = \btrig{}{} \otimes_{\Qp} V$. We next use an analogue of Tate's normalized traces to descend from $\dtrig{}{}(V)$ to a module $\drig{}(V)$ over the Robba ring $\brig{}{,K}$. This is the basic idea of the Colmez-Sen-Tate method \cite{BC08}. 

However, the space $\dtrig{}{}(V)$ is also a $p$-adic representation of $\Gamma_K$, and it is easy to see that $\drig{}(V)$ consists of vectors of $\dtrig{}{}(V)$ that are locally analytic for the action of $\Gamma_K$ (more precisely: pro-analytic, denoted by ${}^{\pa}$) so that $\drig{}(V) \subset \dtrig{}{}(V)^{\pa}$. In fact, by theorem \ref{pgmancy}, we have 
\[ \dtrig{}{}(V)^{\pa} = \cup_{n \geq 0} \phi^{-n}(\drig{}(V)). \]

The construction of the $p$-adic local Langlands correspondence for $\GL_2(\Qp)$ (see for instance \cite{ICBM}, \cite{CGL2} and \cite{LBGL}) uses these cyclotomic $(\phi,\Gamma)$-modules in an essential way. In order to extend this correspondence to $\GL_2(F)$, where $F$ is a finite extension of $\Qp$, it seems necessary to have at our disposal a theory of $(\phi,\Gamma)$-modules for which $\Gamma=\Gal(K_\infty/K)$ where $F \subset K$ and $K_\infty$ is generated by the torsion points of a Lubin-Tate group attached to a uniformizer of $F$. Generalizing the theory of $(\phi,\Gamma)$-modules to higher dimensional $p$-adic Lie groups $\Gamma$ is a difficult problem, which is raised in the introduction to \cite{F90}. It does not seem to be always possible; for example, the main result of \cite{LTFN} implies that under a reasonable additional assumption, $\Gamma$ needs to be abelian for the theory to work. 

If $K_\infty/K$ is a Lubin-Tate extension, then the theory does extend: using Fontaine's classical construction, we can attach to each representation of $G_K$ a Lubin-Tate $(\phi_q,\Gamma_K)$-module over a certain $2$-dimensional local field $\bfont_K$ (\cite{F90} and \cite{KR09}), but the resulting $(\phi,\Gamma)$-modules are usually not overconvergent \cite{FX12}. 

Our solution to the problems of extending the theory and of the lack of overconvergence is to construct $(\phi,\Gamma)$-modules with coefficients in rings of pro-analytic vectors, which is a straightforward generalization of the above observation that $\dtrig{}{}(V)^{\pa} = \cup_{n \geq 0} \phi^{-n}(\drig{}(V))$ in the cyclotomic case. Our main result in this direction is the following (see theorem \ref{lagenok}, and the rest of the article for  notation).

\begin{enonce*}{Theorem A}
If $K_\infty$ is a $p$-adic Lie extension that contains a subextension $K_\infty^\eta$, cut out by an unramified twist $\eta \chicyc$ of the cyclotomic character, and $\btrig{}{,K} = (\btrig{}{})^{H_K}$, then $\dtdag{}_{\rig,K}(V)^{\pa} = (\btrig{}{} \otimes_{\Qp} V)^{H_K,\pa}$ is a free $(\btrig{}{,K})^{\pa}$-module of rank $\dim(V)$, stable under $\phi$ and $\Gamma_K = \Gal(K_\infty/K)$.
\end{enonce*}

The condition on $K_\infty$ is always satisfied if $K_\infty$ is a Lubin-Tate extension. Theorem A allows us to construct $(\phi,\Gamma)$-modules over rings of pro-analytic vectors such as $(\btrig{}{,K})^{\pa}$. It would be interesting to determine the precise structure of these rings. When $K_\infty/K$ is a Lubin-Tate extension, so that $\Gamma_K$ is an open subgroup of $\OO_F^\times$, we show (theorem \ref{btanqp}) that $(\btrig{}{,K})^{\pa}$ contains as a dense subset the ring $\phi^{-\infty}(\calR)$, where $\calR$ is a Robba ring in $[F:\Qp]$ variables (see \cite{ZS}). This is why we call $(\phi,\Gamma)$-modules over $(\btrig{}{,K})^{\pa}$ \emph{multivariable $(\phi,\Gamma)$-modules}.

We next give an application of theorem A to the overconvergence of certain Lubin-Tate $(\phi,\Gamma)$-modules. We first compute the pro-$F$-analytic vectors $(\btrig{}{,K})^{\fpa}$ of $\btrig{}{,K}$ when $K_\infty$ is a Lubin-Tate extension. The following is a corollary of the more precise theorem \ref{xilocan}, where $\brig{}{,K}$ is the Robba ring in one Lubin-Tate variable.

\begin{enonce*}{Theorem B}
If $K_\infty$ is a Lubin-Tate extension, $(\btrig{}{,K})^{\fpa} =\cup_{n \geq 0} \phi_q^{-n}(\brig{}{,K})$.
\end{enonce*}

We also determine enough of the structure of the ring $(\btrig{}{,K})^{\pa}$ in the Lubin-Tate setting to be able to prove a monodromy theorem concerning the descent from $(\btrig{}{,K})^{\pa}$ to $(\btrig{}{,K})^{\fpa}$. We refer to theorem \ref{monoconj} for a precise statement. These results suggest the possibility, in certain cases, of descending $\dtdag{}_{\rig,K}(V)^{\pa}$ to a module over $(\btrig{}{,K})^{\fpa}$ by solving $p$-adic analogues of the Cauchy-Riemann equations. 

Recall now that if $F$ is a finite extension of $\Qp$ contained in $K$, and if $V$ is an $F$-linear representation of $G_K$, then we say that $V$ is $F$-analytic if $\Cp \otimes_F^\tau V$ is the trivial semilinear $\Cp$-representation for all embeddings $\tau : F \to \Qpbar$ with $\tau \neq \Id$. This definition is the natural generalization of Kisin and Ren's notion of $L$-crystalline representations (\S 3.3.7 of \cite{KR09}). Recall that using Fontaine's classical theory, we can attach to each representation of $G_K$ a Lubin-Tate $(\phi_q,\Gamma_K)$-module over a $2$-dimensional local field $\bfont_K$ (\cite{F90} and \cite{KR09}). Using our monodromy theorem, we prove the following result.

\begin{enonce*}{Theorem C}
The Lubin-Tate $(\phi_q,\Gamma_K)$-modules of $F$-analytic representations are overconvergent.
\end{enonce*}

Theorem C was previously known for $F=\Qp$ (Cherbonnier and Colmez \cite{CC98}) and for crystalline representations of $G_K$ (Kisin and Ren \cite{KR09}), as well as for some reducible representations (Fourquaux and Xie \cite{FX12}). Theorem C allows us to attach to each $F$-analytic representation $V$ of $G_K$ an \'etale $F$-analytic Lubin-Tate $(\phi_q,\Gamma_K)$-module $\drig{}(V)$ over $\brig{}{,K}$.

\begin{enonce*}{Theorem D}
The functor $V \mapsto \drig{}(V)$ gives an equivalence of categories between the category of $F$-analytic representations of $G_K$ and the category of \'etale $F$-analytic Lubin-Tate $(\phi_q,\Gamma_K)$-modules over $\brig{}{,K}$.
\end{enonce*}

The question of which Lubin-Tate $(\phi_q,\Gamma_K)$-modules are overconvergent had been around since \cite{CC98} but no significant progress was made for a long time. There were two reasonable possibilities: either all of them were overconvergent, or the $F$-analytic ones were overconvergent. Some unpublished computations of Fourquaux, as well as Kisin and Ren's results, suggested that theorem C was the correct answer. 

\subsection*{Structure of the article} There are ten sections in this article. The first two are devoted to reminders about Lubin-Tate formal groups, and locally analytic vectors. The third one summarizes the construction of many rings of $p$-adic periods, with a Lubin-Tate flavor instead of the traditional cyclotomic one. Sections 4 and 5 are devoted to the analysis of rings of $F$-analytic vectors (theorem \ref{xilocan}, which implies theorem B) and $\Qp$-analytic vectors in some of these rings of $p$-adic periods. Using these results, we prove in section 6 our monodromy theorem (theorem \ref{monoconj}). Section 7 contains reminders about (one variable) Lubin-Tate $(\phi_q,\Gamma_K)$-modules. In section 8, we prove that in the cyclotomic case (and the twisted cyclotomic case), the overconvergent $(\phi,\Gamma)$-modules can be recovered from spaces of locally analytic vectors (theorem \ref{pgmancy}). This idea is used in section 9 to define multivariable $(\phi,\Gamma)$-modules when $K_\infty/K$ is a $p$-adic Lie extension (theorem \ref{lagenok}, which is theorem A). Section 10 contains the proof of theorems C and D. The proofs rely on the constructions of sections 8 and 9, and the monodromy theorem of chapter 6.

\section{Lubin-Tate extensions}
\label{ltext}

Throughout this article, $F$ is a finite extension of $\Qp$ with ring of integers $\OO_F$, uniformizer $\pi$ and residue field $k_F$. Let $q=p^h$ be the cardinality of $k_F$ and let $F_0=W(k_F)[1/p]$. Let $e$ be the ramification index of $F$, so that $eh=[F:\Qp]$. Let $\sigma$ denote the absolute Frobenius map on $F_0$. Let $\emb$ denote the set of embeddings of $F$ in $\Qpbar$. If $\tau \in \emb$, then there exists $n(\tau) \in \ZZ / h \ZZ$ such that $\tau = [x\mapsto x^p]^{n(\tau)}$ on $k_F$. 

Let $\LT$ be a Lubin-Tate formal $\OO_F$-module attached to $\pi$. For $a \in \OO_F$, let $[a](T)$ denote the power series that gives the multiplication-by-$a$ map on $\LT$. We fix a local coordinate $T$ on $\LT$ such that $[\pi](T)=T^q + \pi T$. Let $F_n = F(\LT[\pi^n])$ and let $F_\infty = \cup_{n \geq 1} F_n$. Let $H_F=\Gal(\Qpbar/F_\infty)$ and $\Gamma_F = \Gal(F_\infty/F)$. By Lubin-Tate theory (see \cite{LT65}), $\Gamma_F$ is isomorphic to $\OO_F^\times$ via the Lubin-Tate character $\chilt : \Gamma_F \to \OO_F^\times$.

If $K$ is a finite extension of $F$, let $K_n = K F_n$ and $K_\infty = K F_\infty$ and $H_K=\Gal(\Qpbar/K_\infty)$ and $\Gamma_K=\Gal(K_\infty/K)$. There is an unramified character $\eta : G_F \to \Zp^\times$ such that $\Nm_{F/\Qp}(\chilt) = \eta\chicyc$. Let $K_\infty^\eta = \Qpbar^{\ker \eta\chicyc}$. We have $K_\infty^\eta \subset K_\infty$, and $\eta \chicyc$ identifies $\Gal(K_\infty^\eta / K)$ with an open subgroup of $\Zp^\times$.

Let $\Gamma_n = \Gal(K_\infty/K_n)$ so that $\Gamma_n = \{ g \in \Gamma_K$ such that $\chilt(g) \in 1+\pi^n \OO_F\}$. Let $u_0 = 0$ and for each $n \geq 1$, let $u_n \in \Qpbar$ be such that $[\pi](u_n)=u_{n-1}$, with $u_1 \neq 0$. We have $\vp(u_n) = 1/q^{n-1}(q-1)e$ if $n \geq 1$ and $F_n=F(u_n)$. Let $Q_n(T)$ be the minimal polynomial of $u_n$ over $F$. We have $Q_0(T)=T$, $Q_1(T)=[\pi](T)/T$ and $Q_{n+1}(T)=Q_n([\pi](T))$ if $n \geq 1$. Let $\log_{\LT}(T) = T + \O(\deg \geq 2) \in F\dcroc{T}$ denote the Lubin-Tate logarithm map, which converges on the  open unit disk and satisfies $\log_{\LT}([a](T)) = a  \cdot \log_{\LT}(T)$ if $a\in \OO_F$.  Note that $\log_{\LT}(T) = T \cdot \prod_{k \geq 1} Q_k(T)/\pi$. Let $\exp_{\LT}(T)$ denote the inverse of $\log_{\LT}(T)$.

\begin{prop}
\label{fouper}
If $\tau \in \emb \setminus \{ \Id \}$, then there exists $x_\tau \in \Cp$ such that $g(x_\tau) = x_\tau + \tau(\log \chilt(g))$ if $g \in G_{F^{\Gal}}$.
\end{prop}

\begin{proof}
In \S 3.2 of \cite{FX09}, Fourquaux constructs an element $\xi_\tau$ such that $\xi_\tau / g(\xi_\tau)  = \tau(\chilt(g))$ if $g \in G_{F^{\Gal}}$, so we can take $x_\tau = - \log \xi_\tau$.
\end{proof}

Let $E$ be a finite extension of $F$ that contains $F^{\Gal}$. It is a field of coefficients and we assume that $E/\Qp$ is Galois (this assumption can easily be removed at the expense of more notation). Let $\weil = \{ (\tau,n) \in \Gal(E/\Qp) \times \ZZ$ such that $n(\tau{\mid_F}) \equiv n \bmod{h}\}$. If $w=(\tau,n) \in \weil$, let $n(w) = n$.

\section{Locally analytic and pro-analytic vectors}
\label{ansec}

Let $G$ be a $p$-adic Lie group (in this article, $G$ is most of the time an open subgroup of $\OO_F^\times$) and let $W$ be a $\Qp$-Banach representation of $G$. The space of locally analytic vectors of $W$ is defined in \S 7 of \cite{ST03}. Here we follow the construction given in the monograph \cite{MEAMS}. Let $\NN = \ZZ_{\geq 0}$. Let $H$ be an open  subgroup of $G$ such that there exist coordinates $c_1,\hdots,c_d : H \to \Zp$ giving rise to an analytic bijection $\cbf : H \to \Zp^d$. If $w \in W$, we say that $w$ is an $H$-analytic vector if there exists a sequence $\{w_\kbf\}_{\kbf \in \NN^d}$ with $w_\kbf \to 0$ in $W$, such that $g(w) = \sum_{\kbf \in \NN^d} \cbf(g)^\kbf w_\kbf$ for all $g \in H$. Let $W^{H\dan}$ denote the space of $H$-analytic vectors. This space injects into $\calC^{\an}(H,W)$ and we endow it with the induced norm, which we denote by $\|\cdot\|_H$. We have $\|w\|_H = \sup_{\kbf \in \NN^d} \| w_{\kbf}\|$, and $W^{H\dan}$ is a Banach space. 

We say that a vector $w \in W$ is locally analytic if there exists an open subgroup $H$ as above such that $w \in W^{H\dan}$. Let $W^{\la}$ denote the space of such vectors. We have $W^{\la} = \cup_{H} W^{H\dan}$ where $H$ runs through a sequence of open subgroups of $G$. We endow $W^{\la}$ with the inductive limit topology, so that $W^{\la}$ is an LB space. In the rest of this article, we use the following results, some of which already appear in \S 2.1 of \cite{STLAV}.

\begin{lemm}
\label{ringlocan}
If $W$ is a ring, such that $\|xy\| \leq \|x\| \cdot \|y\|$ if $x,y \in W$, then $W^{H\dan}$ is a ring and $\|xy\|_H \leq \|x\|_H \cdot \|y\|_H$ if $x,y \in W^{H\dan}$.
\end{lemm}

\begin{proof}
Write $g(x) = \sum_{\kbf \in \NN^d} \cbf(g)^\kbf x_\kbf$ and $g(y) = \sum_{\kbf \in \NN^d} \cbf(g)^\kbf y_\kbf$. We have \[ g(xy) = \sum_{\kbf \in \NN^d} \cbf(g)^\kbf \left( \sum_{\ibf+\jbf=\kbf} x_\ibf y_\jbf \right), \]
and if $\ibf+\jbf=\kbf$, then $\| x_\ibf \cdot y_\jbf \| \leq \| x_\ibf \| \cdot \| y_\jbf\|$.
\end{proof}

\begin{prop}
\label{basan}
Let $G$ be a $p$-adic Lie group, let $B$ be a Banach $G$-ring, and let $W$ be a free $B$-module of finite rank, equipped with a compatible action of $G$. If the $B$-module $W$ has a basis $w_1,\hdots,w_d$ in which $g \mapsto \Mat(g)$ is a globally analytic function $G \to \GL_d(B) \subset \mathrm{M}_d(B)$, then 
\begin{enumerate}
\item $W^{H\dan} = \oplus_{j=1}^d B^{H\dan} \cdot w_j$ if $H$ is a subgroup of $G$;
\item $W^{\la} = \oplus_{j=1}^d B^{\la} \cdot w_j$.
\end{enumerate}
\end{prop}

\begin{proof}
This is proved in \S 2.1 of \cite{STLAV}, but we recall the proof for the convenience of the reader. Note that the isomorphism $W \simeq B^d$ provided by the choice of a basis of $W$ is a homeomorphism by the open image theorem. Since (2) follows from (1), it is enough to prove (1). The inclusion $\oplus_{i=1}^d B^{H\dan} \cdot w_i\subset W^{H\dan}$ is obvious, so we show the other inclusion. If $w \in W$, we can write $w= \sum_{i=1}^d b_i w_i$. 
Let $f_i : W \to B$ be the function $w \mapsto b_i$. Write $\Mat(g)=(m_{i,j}(g))_{i,j}$. If $h \in H$, then $h(w) = \sum_{i,j=1}^d h(b_i) m_{i,j}(h) w_j$ and if $w \in W^{H\dan}$, then $h \mapsto f_j(h(w)) = \sum_{i=1}^d h(b_i) m_{i,j}(h)$ is a globally analytic function $H \to B$. If $\Mat(h)^{-1}=(n_{i,j}(h))_{i,j}$, then $h(b_i) = \sum_{j=1}^d f_j(h(w)) n_{i,j}(h)$ so that $b_i \in B^{H\dan}$.
\end{proof}

Let $W$ be a Fr\'echet space, whose topology is defined by a sequence $\{ p_i \}_{i \geq 1}$  of seminorms. Let $W_i$ denote the Hausdorff completion of $W$ for $p_i$, so that $W = \projlim_{i \geq 1} W_i$. The space $W^{\la}$ can be defined if $W$ is a Fr\'echet space (see \cite{MEAMS}), but the resulting object is too small in general, leading us to make to following definition.

\begin{defi}\label{defpav}
If $W= \projlim_{i \geq 1} W_i$ is a Fr\'echet representation of $G$, then a vector $w \in W$ is \emph{pro-analytic} if its image $\pi_i(w)$ in $W_i$ is a locally analytic vector for all $i$. We denote by $W^{\pa}$ the set of such vectors.
\end{defi}

We extend the definition of $W^{\la}$ and $W^{\pa}$ to the cases when $W$ is an LB space and an LF space respectively. Note that if $W$ is an LB space, then $W^{\la} = W^{\pa}$. If  $W$ is an LF  space, then $W^{\la} \subset W^{\pa}$ but $W^{\pa}$ will generally be bigger.

\begin{prop}
\label{baspa}
Let $G$ be a $p$-adic Lie group, let $B$ be a Fr\'echet $G$-ring, and let $W$ be a free $B$-module of finite rank, equipped with a compatible action of $G$. If the $B$-module $W$ has a basis $w_1,\hdots,w_d$ in which $g \mapsto \Mat(g)$ is a pro-analytic function $G \to \GL_d(B) \subset \mathrm{M}_d(B)$, then $W^{\pa} = \oplus_{j=1}^d B^{\pa} \cdot w_j$.\end{prop}

\begin{proof}
If $w \in W$, then we can write $w=\sum_{j=1}^d b_j w_j$ with $b_j \in B$. If $w \in W^{\pa}$ and $i \geq 1$, then $\pi_i(b_j) \in B_i^{\la}$ for all $i$ by proposition \ref{basan}, so that $b_j \in B^{\pa}$.
\end{proof}

The map $\ell : g \mapsto \log_p \chilt(g)$ gives an $F$-analytic isomorphism between $\Gamma_n$ and $\pi^n \OO_F$ for $n \gg 0$. If $W$ is an $F$-linear Banach representation of $\Gamma_K$ and $n \gg 0$, then we say that an element $w \in W$ is $F$-analytic on $\Gamma_n$ if there exists a sequence $\{w_k\}_{k \geq 0}$ of elements of $W$ with $\pi^{nk} w_k \to 0$ such that $g(w) = \sum_{k \geq 0} \ell(g)^k w_k$ for all $g \in \Gamma_n$. Let $W^{\Gamma_n\dan,\fla}$ denote the space of such elements. Let $W^{\fla} = \cup_{n \geq 1} W^{\Gamma_n\dan,\fla}$. 

\begin{lemm}
\label{danfla}
We have $W^{\Gamma_n\dan,\fan} = W^{\Gamma_n\dan} \cap W^{\fan}$.
\end{lemm}

\begin{proof}
It is clear that $W^{\Gamma_n\dan,\fan} \subset W^{\Gamma_n\dan} \cap W^{\fan}$, so we prove the reverse inclusion. If $w \in W^{\fan}$, then we have $g(w) = \sum_{k \geq 0} \ell(g)^k w_k$ for $g \in \Gamma_m$ with $m \gg 0$. Let $c_1,\hdots,c_d$ be a set of coordinates on $\Gamma_n$, so that $c : \Gamma_n \to \Zp^d$ is a bijection. There exist $\lambda_1,\hdots,\lambda_d \in \OO_F$ such that $\ell(g) = \sum_{i=1}^d \lambda_i c_i(g)$. We have 
\[ g(w) = \sum_{k \geq 0} \sum_{|\jbf|=k} \cbf (g)^{\jbf} \mathbf{\lambda}^{\jbf} \binom{k}{\jbf} w_k \quad\text{if $g \in \Gamma_m$}. \]
If $g \in W^{\Gamma_n\dan}$, then $\mathbf{\lambda}^{\jbf} \binom{|\jbf|}{\jbf} w_{|\jbf|} \to 0$ as $|\jbf| \to +\infty$. Since $\ell : \Gamma_n \to \pi^n \OO_F$ is an isomorphism, at least one of the $\lambda_i$ belongs to $\pi^n \OO_F \setminus \pi^{n+1} \OO_F$. Taking $\jbf$ to be $k$ at the $i$-th entry and $0$ at the others, we get that $\pi^{nk} w_k \to 0$, so that $w \in W^{\Gamma_n\dan,\fan}$.
\end{proof}

Finally, recall the following simple result (\S 2.1 of \cite{STLAV}).

\begin{lemm}
\label{eventnorm}
If $w \in W^{\la}$, then $\|w\|_{\Gamma_m} = \|w\|$ for $m \gg 0$.
\end{lemm}

Recall that $E$ is a field that contains $F^{\Gal}$. If $\tau \in \emb$, then we have the derivative in the direction $\tau$, which is an element $\nabla_\tau \in E \otimes_{\Qp} \Lie(\Gamma_F)$. It can be constructed in the following way (see also \S 3.1 of \cite{DI12}). The $E$-vector space $\Hom_{\Qp}(F,E)$ is generated by the elements of $\emb$. If $W$ is an $E$-linear Banach representation of $\Gamma_K$ and if $w \in W^{\la}$ and $g \in \Gamma_K$, then there exists elements $\{ \nabla_\tau \}_{\tau \in \emb}$ of $F^{\Gal} \otimes_{\Qp} \Lie(\Gamma_F)$ such that we can write 
\[ \log g (w) = \sum_{\tau \in \emb} \tau(\ell(g)) \cdot \nabla_\tau(w). \]

With the same notation, there exist $m \gg 0$ and elements $\{w_\kbf \}_{\kbf \in \NN^{\emb}}$ such that if $g \in \Gamma_m$, then $g (w) = \sum_{\kbf \in \NN^{\emb}} \ell(g)^\kbf w_\kbf$, where $\ell(g)^\kbf = \prod_{\tau \in \emb} \tau \circ \ell(g)^{k_\tau}$. We have $\nabla_\tau(w) = w_{\unbf_\tau}$ where $\unbf_\tau$ is the $\emb$-tuple whose entries are $0$ except the $\tau$-th one which is $1$. If $\kbf \in\NN^{\emb}$, and if we set  $\nabla^\kbf(w) =  \prod_{\tau \in \emb} \nabla_{\tau}^{k_\tau} (w)$, then $w_\kbf = \nabla^\kbf(w)/\kbf!$. 

\begin{rema}
\label{flanab}
If $w \in W^{\la}$, then $w \in W^{\fla}$ if and only if $\nabla_\tau(w) = 0$ for all $\tau \in \emb \setminus \{ \Id \}$.
\end{rema}

\begin{lemm}
\label{hnabtau}
If $h \in \Gal(E/\Qp)$ acts on $E \otimes_{\Qp} \Lie(\Gamma_F)$, then $h(\nabla_\tau) = \nabla_{h \circ \tau}$. In particular, $\nabla_{\Id} \in F \otimes_{\Qp} \Lie(\Gamma_F)$.
\end{lemm}

\begin{proof}
Let $g_1,\hdots,g_d$ be elements of $\Gamma_K$ such that $\log g_1, \hdots, \log g_d$ generate $\Lie(\Gamma_F)$. Write $\nabla_\tau = \sum_{i=1}^d n_i^\tau \cdot \log g_i$ with $n_i^\tau \in E$. Applying the definition of $\nabla_\tau$ to $g=g_j$ we find that $\log g_j (w) = \sum_{\tau \in \emb}  \sum_{i=1}^d n_i^\tau \cdot \tau(\ell(g_j)) \cdot \log g_i$ so that $\sum_{\tau \in \emb}  \sum_{i=1}^d n_i^\tau \cdot \tau(\ell(g_j)) = \delta_{ij}$. These equations determine $n_i^\tau$. In particular, applying $h$ gives $\sum_{\tau \in \emb}  \sum_{i=1}^d h(n_i^\tau) \cdot h \circ \tau(\ell(g_j)) = \delta_{ij}$ so that $h(n_i^\tau) = n_i^{h \circ \tau}$. This implies the first assertion. The second follows from applying the first to all $h \in \Gal(E/F)$.
\end{proof}

\begin{lemm}\label{twistder}
Let $X$ and $Y$ be $E$-representations of $\Gamma_n$, take $g \in \Gal(E/\Qp)$, and let $f : X \to Y$ be a $\Gamma_n$-equivariant map such that $f(ax)=g^{-1}(a)f(x)$ for $a \in E$. If $x \in X^{\pa}$, then $\nabla_{\Id}(f(x)) = f( \nabla_{g{\mid_F}} (x))$.
\end{lemm}

The space $\hat{K}_\infty$ is a Banach representation of $\Gamma_K$ which is studied in \cite{STLAV}.

\begin{prop}
\label{nabnul}
We have $\nabla_{\Id} = 0$ on $(\hat{K}_\infty)^{\la}$, and $(\hat{K}_\infty)^{\fla} = K_\infty$.
\end{prop}

\begin{proof}
This is corollary 4.3 of \cite{STLAV}. However, \cite{STLAV} is written under the assumption that $F/\Qp$ is Galois, and we now prove the result without this assumption. By \S 6 of ibid., there is a nonzero element $\mathfrak{a} \in \Cp \otimes_{\Qp} \Lie(\Gamma_K)$ such that $\mathfrak{a} = 0$ on $(\hat{K}_\infty)^{\la}$. Write $\mathfrak{a} = \sum_{\tau \in \emb} a_{\tau} \nabla_\tau$. For each $\tau \in \emb \setminus \{ \Id \}$, proposition \ref{fouper} provides us with an element $x_\tau \in (\hat{K}_\infty)^{\la}$ such that $\nabla_\tau(x_\tau) = 1$ and $\nabla_\upsilon(x_\tau) = 0$ if $\tau \neq \upsilon$. Applying $\mathfrak{a}$ to $x_\tau$, we get $a_\tau = 0$ if $\tau \in \emb \setminus \{ \Id \}$. This implies that $\nabla_{\Id} = 0$ on $(\hat{K}_\infty)^{\la}$. The second assertion follows immediately.
\end{proof}

\section{Rings of $p$-adic periods}
\label{ltper}

In this section, we recall the definition of a number of rings of $p$-adic periods. These definitions can be found in \cite{F90,FPP} and \cite{LB2}, but we also use the Lubin-Tate generalization given for instance in \S\S 8,9 of  \cite{C02}. We therefore assume that $F_\infty$ is the Lubin-Tate extension of $F$ constructed above. Let $\etplus = \{ (x_0,x_1,\hdots)$, with $x_n \in \OO_{\Cp}/\pi$ and $x_{n+1}^q=x_n$ for all $n \geq 0\}$. This ring is endowed with the valuation $\ve(\cdot)$ defined by $\ve(x) = \lim_{n \to +\infty} q^n \vp(\hat{x}_n)$ where $\hat{x}_n \in \OO_{\Cp}$ lifts $x_n$. The ring $\etplus$ is complete for $\ve(\cdot)$. If the $\{u_n\}_{n \geq 0}$ are as in section \ref{ltext}, then $\ubar = (\ubar_0,\ubar_1,\hdots) \in \etplus$ and $\ve(\ubar) = q/(q-1)e$. Let $\et$ be the fraction field of $\etplus$.

Let $W_F(\cdot)$ denote the functor $\OO_F \otimes_{\OO_{F_0}} W(\cdot)$ of $F$-Witt vectors. Let $\atplus=W_F(\etplus)$ and let $\btplus=\atplus[1/\pi]$. These rings are preserved by the Frobenius map $\phi_q=\Id \otimes \phi^h$. Every element of $\btplus[1/[\ubar]]$ can be written as $\sum_{k \gg -\infty} \pi^k [x_k]$ where $\{x_k\}_{k \in \ZZ}$ is a bounded sequence of $\et$. If $r \geq 0$, define a valuation\footnote{This valuation is normalized as in \S 2 of \cite{LB2}. The valuation defined in \S 3 of \cite{PGMLT} is normalized differently (sorry), it is $pr/(p-1)$ times this one.}
$V(\cdot,r)$ on $\btplus[1/[\ubar]]$ by 
\[ V(x,r) = \inf_{k \in \ZZ} \left(\frac{k}{e} + \frac{p-1}{pr} \ve(x_k)\right)  \text{ if } x = \sum_{k \gg -\infty} \pi^k [x_k]. \]
If $I$ is a closed subinterval of $[0;+\infty[$, then let $V(x,I) = \inf_{r \in I} V(x,r)$. The ring $\bt^I$ is defined to be the completion\footnote{When $F=\Qp$, the ring $\bt^I$ is the same as the one denoted by $\bt_I$ in \S 2.1 of \cite{LB2}.} of $\btplus[1/[\ubar]]$ for the valuation $V(\cdot,I)$ if $0 \notin I$. If $I=[0;r]$, then $\bt^I$ is the completion of $\btplus$ for $V(\cdot,I)$. Let $\at^I$ be the ring of integers of $\bt^I$ for $V(\cdot,I)$. 

If $k \geq 1$, then let $r_k=p^{kh-1}(p-1)$. The map $\theta \circ \phi_q^{-k} : \atplus \to \OO_{\Cp}$ extends by continuity to $\at^I$ provided that $r_k \in I$ and then $\theta \circ \phi_q^{-k}(\at^I) \subset \OO_{\Cp}$. By \S 9.2 of \cite{C02}, there exists $u \in \atplus$, whose image in $\etplus$ is $\ubar$, and such that $\phi_q(u) = [\pi](u)$ and $g(u) = [\chilt(g)](u)$ if $g \in \Gamma_F$. For $k \geq 0$, let $Q_k=Q_k(u) \in \atplus$. Let $\tilde{\pi} \in \etplus$ be a compatible sequence of $q^n$-th roots of $\pi$. The kernel of $\theta : \atplus \to \OO_{\Cp}$ is generated by $\phi_q^{-1}(Q_1)$ or by $[\tilde{\pi}]-\pi$ (see proposition 8.3 of \cite{C02}), so that $\phi_q^{-1}(Q_1) / ([\tilde{\pi}]-\pi)$ is a unit of $\atplus$ and therefore, $Q_k / ([\tilde{\pi}]^{q^k}-\pi)$ is a unit of $\atplus$ for all $k \geq 1$.

\begin{lemm}
\label{atzs}
If $y \in \at^{[0;r_k]}$, then there exists a sequence $\{a_i\}_{i \geq 0}$ of elements of $\atplus$, converging $p$-adically to $0$, such that $y=\sum_{i \geq 0} a_i \cdot (Q_k/\pi)^i$.
\end{lemm}

\begin{proof}
Let $F=\Qp$ and $\pi=p$. In \S 2.1 of \cite{LB2}, it is proved that we can write $y=\sum_{i \geq 0} b_i \cdot ([\tilde{\pi}]^{q^k}/\pi)^i$ with $b_i \in \atplus$ and $b_i \to 0$. We can therefore write $y=\sum_{i \geq 0} a_i \cdot (([\tilde{\pi}]^{q^k}-\pi)/\pi)^i$ with $a_i \in \atplus$ and $a_i \to 0$. The lemma  follows from the fact recalled above that $Q_k / ([\tilde{\pi}]^{q^k}-\pi)$ is a unit of $\atplus$. The proof for general $F$ and $\pi$ is similar. 
\end{proof}

\begin{lemm}\label{propatrs}
Let $r=r_\ell$ and $s=r_k$, with $1 \leq \ell \leq k$.
\begin{enumerate}
\item\label{it1}  $\theta \circ \phi_q^{-k}(\at^{[r;s]}) = \OO_{\Cp}$ and $\ker(\theta \circ \phi_q^{-k} : \at^{[r;s]} \to \OO_{\Cp}) = (Q_k/\pi) \cdot \at^{[r;s]}$;
\item\label{it2} $\pi \at^{[r;s]} \cap (Q_k/\pi) \cdot \at^{[r;s]} = Q_k \cdot \at^{[r;s]}$;
\item\label{it3} $\pi \at^{[r;s]} \cap \at^{[0;s]} = \pi \at^{[0;s]}$.
\end{enumerate}
\end{lemm}

\begin{proof}
Item (\ref{it1}) follows from the straightforward generalization of \S 2.2 of \cite{LB2} from $\Qp$ to $F$ (note that proposition 2.11 of ibid.\ is only correct if the element $[\tilde{p}]/p-1$ actually belongs to $\at^I$) and the fact that $Q_k / ([\tilde{\pi}]^{q^k}-\pi)$ is a unit of $\atplus$. If $x \in \at^{[r;s]}$ and $\pi x \in \ker(\theta \circ \phi_q^{-k})$, then $x \in \ker(\theta \circ \phi_q^{-k})$ and this together with (\ref{it1}) implies (\ref{it2}). Finally, if $x \in \at^{[r;s]}$ is such that $\pi x \in \at^{[0;s]}$, then $x \in \bt^{[0;s]}$ and $V(x,s) \geq V(x,[r;s]) \geq 0$, so that $x \in \at^{[0;s]}$ and $\pi x \in \pi \at^{[0;s]}$.
\end{proof}

\begin{prop}\label{liftdivp}
If $y \in \at^{[0;s]} + \pi \cdot \at^{[r;s]}$ and if $\{y_i\}_{i \geq 0}$ is a sequence of elements of $\atplus$ such that $y-\sum_{j=0}^{j-1} y_i \cdot (Q_k / \pi)^i$ belongs to $\ker(\theta \circ \varphi_q^{-k})^j$ for all $j \geq 1$, then there exists $j \geq 1$ such that $y-\sum_{i=0}^{j-1} y_i \cdot (Q_k / \pi)^i \in \pi \cdot \at^{[r;s]}$.
\end{prop}

\begin{proof}
By lemma \ref{atzs}, there exist $j \geq 1$ and $a_0,\hdots,a_{j-1}$ of $\atplus$ such that
\begin{equation}\label{dif}\tag{A}
y - \left(a_0 + a_1 \cdot (Q_k/\pi) + \cdots + a_{j-1} \cdot (Q_k/\pi)^{j-1}\right) \in \pi \at^{[r;s]}.
\end{equation}
We have $a_0, y_0 \in \atplus$ and $\theta \circ \phi_q^{-k}(y_0 - a_0) \in \pi \OO_{\Cp}$ by the above, so there exists $c_0,d_0 \in \atplus$ such that $a_0 = y_0 + Q_k c_0 + \pi d_0$. In particular, \eqref{dif} holds if we replace $a_0$ by $y_0$. Assume now that $f \leq j-1$ is such that \eqref{dif} holds if we replace $a_i$ by $y_i$ for $i \leq f-1$. The element
\begin{multline*} 
\left(a_0 + a_1 \cdot (Q_k/\pi) + \cdots + a_{j-1} \cdot (Q_k/\pi)^{j-1}\right)  \\ - \left(y_0 + y_1 \cdot (Q_k/\pi) + \cdots + y_{j-1} \cdot (Q_k/\pi)^{j-1}\right)
\end{multline*}
belongs to $\pi \at^{[r;s]} + (Q_k/\pi)^j  \at^{[r;s]}$. If $a_i = y_i$ for $i \leq f-1$, then the element
\begin{multline*}
\left(a_f + a_{f + 1} \cdot (Q_k/\pi) + \cdots + a_{j-1} \cdot (Q_k/\pi)^{j-1-f}\right) \\ - \left(y_f + y_{f + 1} \cdot (Q_k/\pi) + \cdots + y_{j-1} \cdot (Q_k/\pi)^{j-1-f}\right)
\end{multline*}
belongs to $\pi \at^{[r;s]} + (Q_k/\pi)^{j-f}  \at^{[r;s]}$ since $\pi \at^{[r;s]} \cap (Q_k/\pi)^f  \at^{[r;s]} = \pi (Q_k/\pi)^f  \at^{[r;s]}$ by repeatedly applying (\ref{it2}) of lemma \ref{propatrs}. We have $a_f,y_f \in \atplus$ and the above implies that $\theta \circ \phi_q^{-k}(y_f - a_f) \in \pi \OO_{\Cp}$. There exist therefore $c_f,d_f \in \atplus$ such that $a_f = y_f + Q_k c_f + \pi d_f$ which shows that \eqref{dif} holds if we also replace $a_f$ by $y_f$. This shows by induction on $f$ that $y - (y_0 + y_1 \cdot (Q_k/\pi) + \cdots + y_{j-1} \cdot (Q_k/\pi)^{j-1})$ belongs to $\pi \at^{[r;s]}$, which proves the proposition.
\end{proof}

\begin{lemm}\label{ber31}
If $r>1$, then $u/[\ubar]$ is a unit of $\at^{[r;+\infty]}$.
\end{lemm}

\begin{proof}
We have $u = [\ubar] + \sum_{k\geq 1} \pi^k [v_k]$ with $v_k \in \etplus$ and the lemma follows from the fact that if $s \geq r > 1 \geq (p-1)/p \cdot q/(q-1)$, then $V(\pi/[\ubar],s) > 0$.
\end{proof}

For $\rho>0$, let $\rho'=\rho \cdot e \cdot p/(p-1) \cdot (q-1)/q$. Lemma \ref{ber31} and the fact that $\ve(\ubar)=q/(q-1)e$ imply that if $r>1$, then\footnote{compare with proposition 3.1 of \cite{PGMLT}, bearing in mind that our normalization of $V(\cdot,r)$ is different.}  $V(u^i,r) = i/r'$ for $i \in \ZZ$.

Let $I$ be either a subinterval of $]1;+\infty[$ or such that $0 \in I$, and let $f(Y) = \sum_{k \in \ZZ} a_k Y^k$ be a power series with $a_k \in F$ and such that $\vp(a_k)+k/\rho' \to +\infty$ when $|k| \to +\infty$ for all $\rho \in I$. The series $f(u)$ converges in $\bt^I$ and we let $\bfont_F^I$ denote the set of $f(u)$ where $f(Y)$ is as above. It is a subring of $\bt_F^I = (\bt^I)^{H_F}$, which is stable under the action of $\Gamma_F$. The Frobenius map gives rise to a map $\phi_q : \bfont_F^I \to \bfont_F^{qI}$. If $m \geq 0$, then $\phi_q^{-m}(\bfont_F^{q^m I}) \subset \bt_F^I$ and we let $\bfont_{F,m}^I = \phi_q^{-m}(\bfont_F^{q^m I})$ so that $\bfont_{F,m}^I \subset \bfont_{F,m+1}^I$ for all $m \geq 0$. For example, if we let $t_\pi = \log_{\LT}(u)$ then $t_\pi \in \bfont_F^{[0;+\infty[}$, and $\phi_q(t_\pi) = \pi t_\pi$ and $g(t_\pi) = \chilt(g) t_\pi$ if $g \in G_F$.

Let $\brig{,r}{,F}$ denote the ring $\bfont_F^{[r;+\infty[}$.  This is a subring of $\bt_F^{[r;s]}$ for all $s \geq r$. Let $\bdag{,r}_F$ denote the set of $f(u) \in \brig{,r}{,F}$ such that in addition $\{a_k\}_{k \in \ZZ}$ is a bounded sequence. Let $\bdag{}_F = \cup_{r \gg 0} \bdag{,r}_F$. This a Henselian field (cf.\  \S 2 of \cite{SM95}), whose residue field $\efont_F$ is isomorphic to $\FF_q \dpar{\overline{u}}$. Let $K$ be a finite extension of $F$. By the theory of the field of norms (see \cite{FW1,FW2} and \cite{WCN}), there corresponds to $K/F$ a separable extension $\efont_K / \efont_F$, of degree $[K_\infty:F_\infty]$. Since $\bdag{}_F$ is a Henselian field, there exists a finite unramified extension $\bdag{}_K/ \bdag{}_F$ of degree $f = [K_\infty:F_\infty]$ whose residue field is $\efont_K$ (cf.\  \S 3 of \cite{SM95}). There exist therefore $r(K) >0$ and elements $x_1,\hdots,x_f$ in $\bdag{,r(K)}_K$ such that $\bdag{,s}_K = \oplus_{i=1}^f \bdag{,s}_F \cdot x_i$ for all $s \geq r(K)$. If $r(K) \leq \min (I)$, then let $\bfont^I_K$ denote the completion of $\bdag{,r(K)}_K$ for $V(\cdot,I)$, so that $\bfont^I_K = \oplus_{i=1}^f \bfont^I_F \cdot x_i$. Let $\bfont^I_{K,m} = \phi_q^{-m}(\bfont^{q^m I}_K)$ and $\bfont^I_{K,\infty} = \cup_{m \geq 0} \bfont^I_{K,m}$ so that $\bfont^I_{K,m} \subset \bt_K^I = (\bt^I)^{H_K}$. 

Let $\brig{,r}{,K}$ denote the Fr\'echet completion of $\bdag{,r}_K$ for the valuations $\{V(\cdot,[r;s]) \}_{s \geq r}$. Let $\brig{,r}{,K,m} = \phi_q^{-m}(\brig{,q^m r}{,K})$ and $\brig{,r}{,K,\infty} = \cup_{m \geq 0} \brig{,r}{,K,m}$. We have $\brig{,r}{,K,\infty} \subset \bt_K^{[r;s]}$ for all $s \geq r$. Let $\btrig{,r}{}$ denote the Fr\'echet completion of $\btplus[1/[\ubar]]$ for the valuations $\{V(\cdot,[r;s]) \}_{s \geq r}$; $\btrig{,r}{}$ is a subring of $\bt^{[r;s]}$ for all $s \geq r$. Let $\btrig{}{} = \cup_{r \gg 0} \btrig{,r}{}$ and $\btrig{,r}{,K} = (\btrig{,r}{})^{H_K}$ and $\btrig{}{,K} = (\btrig{}{})^{H_K}$. Note that $\btrig{,r}{,K}$ contains $\brig{,r}{,K}$. We also write $\bt^{[r;+\infty[}_K$ instead of $\btrig{,r}{,K}$ and $\bfont^{[r;+\infty[}_K$ instead of $\brig{,r}{,K}$.

\section{Locally $F$-analytic vectors of $\btrig{}{,K}$}
\label{fansec}

In this section, we still assume that $K_\infty$ is the Lubin-Tate extension of $K$ constructed above, and we compute the pro-$F$-analytic vectors of $\btrig{}{,K}$. Recall that if $n \geq 1$, then we set $r_n=p^{nh-1}(p-1)$, so that $r'_n=q^{n-1}(q-1)\cdot e$. From now on, let $r=r_\ell$ and $s=r_k$, with $\ell \leq k$. Let $I=[r;s]$ with $r=r_\ell$ and $s=r_k$.

\begin{prop}\label{radexp}
If $f(Y) \in \OO_F[Y]$, then $\phi_q^{-m}(f(u)) \in (\bt_F^I)^{\Gamma_{m+k}\dan,\fla}$.
\end{prop}

\begin{proof}
By lemma \ref{ringlocan}, it is enough to show that $\phi_q^{-m}(u) \in (\bt_F^I)^{\Gamma_{m+k}\dan,\fla}$. By Lubin-Tate theory, there exists a family $\{c_n(T)\}_{n \geq 0}$ of elements of $F[T]$ such that $[a](T) = \sum_{n \geq 0} c_n(a) \cdot T^n$ if $a \in \OO_F$ so that if $g \in \Gamma_F$, then 
\[ g(\phi_q^{-m}(u)) = \phi_q^{-m}(g(u)) = \phi_q^{-m}(\sum_{n \geq 0} c_n(\chilt(g)) \cdot u^n) = \sum_{n \geq 0} c_n(\chilt(g)) \cdot \phi_q^{-m}(u)^n. \]
The polynomials $c_n(T)$ are of degree at most $n$ and $c_n(\OO_F) \subset \OO_F$. 

For $n \geq 0$, let $n_0+n_1 q + \cdots + n_{f-1} q^{f-1}$ denote the representation of $n$ in base $q$. Let $h=k+m$, let $w_{n,h} = \sum_{i=h}^{f-1} n_i \cdot (q^{i-h} - 1)/(q-1)$ (the sum is equal to $0$ if it is empty). Recall that de Shalit has constructed in \cite{dSMB} a family of polynomials $\{ g_n(T) \}_{n \geq 0}$ with the following properties:
\begin{enumerate}
\item $g_n(T)$ belongs to $F[T]$, is of degree $n$, and $g_n(\OO_F) \subset \OO_F$ (\S 1.2 of \cite{dSMB});
\item the family $\{ g_n(T) \}_{n \geq 0}$ is a Mahler basis (\S 1.2 of ibid);
\item the elements $\{\pi^{w_{n,h}} g_n\}_{n \geq 0}$ form a Banach basis of the Banach space $\LA_h(\OO_F)$ of functions on $\OO_F$ that are analytic on closed disks of radius $|\pi|^h$ (\S\S 4.2-4.3 of ibid, see also \S 10 of \cite{YA63}).
\end{enumerate}

Recall that $g(\phi_q^{-m}(u)) = \sum_{n \geq 0} c_n(\chilt(g)) \cdot \phi_q^{-m}(u)^n$. In order to prove the proposition, we need to show that the map $\OO_F \to \bt_F^I$ given by $a \mapsto \sum_{n \geq 0} c_n(a) \cdot \phi_q^{-m}(u)^n$ is analytic on closed disks of radius $|\pi|^h$. Since $c_n(\OO_F) \subset \OO_F$ and the family $\{ g_n(T) \}_{n \geq 0}$ is a Mahler basis, there are elements $b_{n,i} \in \OO_F$ such that $c_n(T)=\sum_{i=0}^n b_{n,i} g_i(T)$. It is therefore enough to show that $\| g_n \|_{\LA_h(\OO_F)} \cdot \| \phi_q^{-m}(u)^n \|_I \to 0$ as $n \to +\infty$, so that
the series of functions $\sum_{n \geq 0} c_n(\cdot) \cdot \phi_q^{-m}(u)^n$ converges in $\LA_h(\OO_F,\bt^I_F)$. We have 
\[ w_{n,h} = \sum_{i=h}^{f-1} n_i  \frac{q^{i-h} - 1}{q-1} \leq \sum_{i=h}^{f-1} n_i  \frac{q^{i-h}}{q-1} \leq n \cdot \frac{1}{q^h(q-1)}. \]

Let $\|{\cdot}\|_s$ denote the norm on $\bt^I$ given by $\|x\|_s = p^{-V(x,s)}$. By the maximum modulus principle (corollary 2.20 of \cite{LB2}), we have $\|\phi_q^{-m}(u)\|_I = \|\phi_q^{-m}(u)\|_s$. We have $\| \phi_q^{-m}(u)^n \|_{r_k} = \| u^n \|_{r_{k+m}} = |\pi|^{n/q^{h-1}(q-1)}$ since $V(u^n,r_h) = n/(q^{h-1}(q-1)\cdot e)$. This implies that 
\[ \| g_n \|_{\LA_h(\OO_F)} \cdot \| \phi_q^{-m}(u)^n \|_s \leq |\pi|^{n \cdot \left(\frac{1}{q^{h-1}(q-1)}-\frac{1}{q^h(q-1)}\right)}, \]
so that $\| g_n \|_{\LA_h(\OO_F)} \cdot \| \phi_q^{-m}(u)^n \|_s \to 0$ as $n \to +\infty$.
\end{proof}

\begin{rema}
\label{radold}
In a previous version of this article, proposition \ref{radexp} was proved under the assumption that the ramification index of $F$ was at most $p-1$, by bounding the norm of $\nabla^i(f) / i!$ as $i \to +\infty$. I am grateful to Pierre Colmez for suggesting the above proof.
\end{rema}

Let $m_0 \geq 0$ be such that $t_\pi$ and $t_\pi/Q_k$ belong to $(\bt^I_F)^{\Gamma_{m_0}\dan,\fla}$.

\begin{lemm}\label{divfanat}
If $m\geq m_0$, $a \in \bt_F^I$ and $Q_k \cdot a \in (\bt^I_F)^{\Gamma_m\dan,\fla}$, then $a \in (\bt^I_F)^{\Gamma_m\dan,\fla}$.
\end{lemm}

\begin{proof}
We have $t_\pi a = t_\pi / Q_k \cdot Q_k \cdot a \in (\bt^I_F)^{\Gamma_m\dan,\fla}$. The derivatives of $t_\pi a$ are all divisible by $t_\pi$ since $\nabla(t_\pi x) = t_\pi x + t_\pi \nabla(x)$, so that we can write $g(t_\pi a) = \sum_{k \geq 0} \ell(g)^k t_\pi a_k$ for $g \in \Gamma_m$. We have $g(a) = \chilt(g)^{-1} \cdot \sum_{k \geq 0} \ell(g)^k a_k$ so that $a \in (\bt^I_F)^{\Gamma_m\dan,\fla}$.
\end{proof}

\begin{theo}\label{xilocan}
Let $I=[r_\ell;r_k]$ with $\ell \leq k$, let $K$ be a finite extension of $F$, and let $m \geq 0$ be such that $t_\pi$ and $t_\pi/Q_k$ belong to $(\bt^I_F)^{\Gamma_{m+k}\dan,\fla}$.
\begin{enumerate}
\item $(\bt^I_F)^{\Gamma_{m+k}\dan,\fla} \subset \bfont^I_{F,m}$;
\item $(\bt^I_K)^{\fla} = \bfont^I_{K,\infty}$;
\item $(\btrig{,r_\ell}{,K})^{\fpa} = \brig{,r_\ell}{,K,\infty}$.
\end{enumerate}
\end{theo}

\begin{proof}
We first prove (1). Take $x \in (\bt^{[r;s]}_F)^{\Gamma_{m+k}\dan,\fla}$. If $d=q^{\ell-1}(q-1)$, then $\at^{[r;s]} =  \at^{[0;s]} \{ \pi/u^d \}$ (see corollary 2.2 of \cite{LB2} for the case $\pi=p$, the generalization is straightforward) so that for all $n \geq 1$, there exists $k_n \geq 0$ such that $(u^d/\pi)^{k_n} \cdot x \in \at^{[0;s]} + \pi^n \at^{[r;s]}$. If $x_n = (u^d/\pi)^{k_n} \cdot x$, then $x_n \in (\bt^{[r;s]}_F)^{\Gamma_{m+k}\dan,\fla}$, so that $\theta \circ \phi_q^{-k}(x_n) \in \OO_{\hat{F}_\infty}^{\Gamma_{m+k}\dan,\fla}$. By proposition \ref{nabnul}, $\hat{F}_\infty^{\fla}=F_\infty$ and therefore, $\OO_{\hat{F}_\infty}^{\Gamma_{m+k}\dan,\fla} = \OO_{F_{m+k}}$.

There exists $y_{n,0} \in \OO_F[\phi_q^{-m}(u)]$ such that $\theta \circ \phi_q^{-k}(x_n)= \theta \circ \phi_q^{-k}(y_{n,0})$. By (\ref{it1}) of lemma \ref{propatrs} and lemma \ref{divfanat}, there exists $x_{n,1} \in (\bt^{[r;s]}_F)^{\Gamma_{m+k}\dan,\fla} \cap \at^{[r;s]}$ such that $x_n-y_{n,0} = (Q_k/\pi) \cdot x_{n,1}$. Applying this procedure inductively gives us a sequence $\{y_{n,i}\}_{i \geq 0}$ of elements of $\OO_F[\phi_q^{-m}(u)]$ such that for all $j \geq 1$, we have
\[ x_n - \left( y_{n,0} + y_{n,1} \cdot (Q_k/\pi) + \cdots + y_{n,j-1} \cdot (Q_k/\pi)^{j-1} \right) \in \ker(\theta \circ \varphi_q^{-k})^j.  \] 
Proposition \ref{liftdivp} shows that there exists $j \gg 0$ such that
\[ x_n - \left( y_{n,0} + y_{n,1} \cdot (Q_k/\pi) + \cdots + y_{n,j-1} \cdot (Q_k/\pi)^{j-1} \right) \in \pi\at^{[r;s]},  \] 
and therefore belongs to $\pi(\at^{[0;s]} + \pi^{n-1} \at^{[r;s]})$, since $\pi\at^{[r;s]} \cap \at^{[0;s]} = \pi \at^{[0;s]}$ by (3) of lemma \ref{propatrs}. Write $x_n - ( y_{n,0} + y_{n,1} \cdot (Q_k/\pi) + \cdots + y_{n,j-1} \cdot (Q_k/\pi)^{j-1}) = \pi x'_n$ with $x'_n \in \at^{[0;s]} + \pi^{n-1} \at^{[r;s]}$. By proposition \ref{radexp}, we have $x'_n \in (\bt^{[r;s]}_F)^{\Gamma_{m+k}\dan,\fla}$. Applying to $x'_n$ the same procedure which we have applied to $x_n$, and proceeding inductively, we  find $j \gg 0$ and elements $\{ y_{n,i} \}_{i \leq j}$ of $\OO_F[\phi_q^{-m}(u)]$ such that if
\[ y_n = y_{n,0} + y_{n,1} \cdot (Q_k/\pi) + \cdots + y_{n,j-1} \cdot (Q_k/\pi)^{j-1}, \] 
then $y_n-x_n \in \pi^n \at^{[r;s]}$. If $z_n = (\pi/u^d)^{k_n} y_n$, then $z_n-x = (\pi/u^d)^{k_n} (y_n-x_n) \in \pi^n \at^{[r;s]}$ so that $\{z_n\}_{n \geq 1}$ converges $\pi$-adically to $x$, and $z_n \in \afont^{[r;s]}_{F,m}$ so that $x \in \afont^{[r;s]}_{F,m}$. 

This proves (1) when $K=F$. The action of $\Gamma_F$ on $\bfont_{F,m}^I$ is locally $F$-analytic, so that $\bfont_{F,\infty}^I \subset (\bt_F^I)^{\fla}$ and this finishes the proof of (1), and the proof of (2) for $K=F$.

We now consider (2) when $K$ is a finite extension of $F$. We first prove that $\bfont^I_{K,\infty} \subset (\bt^I_K)^{\fla}$. Since $\bfont^I_K = \oplus_{i=1}^f \bfont^I_F \cdot x_i$ as at the end of \S \ref{ltper}, each element of $\bfont^I_{K,\infty}$ is integral over $\bfont^I_{F,\infty}$. Take $x$ in $\bfont^I_{K,\infty}$ and let  $P(T) \in \bfont^I_{F,\infty}[T]$ denote its minimal polynomial over $\bfont^I_{F,\infty}$. If $g \in \Gamma_K$ is close enough to $1$, then $(gP)(gx)=0$ and the coefficients of $gP$ are analytic functions in $\ell(g)$. We also have $P'(x) \neq 0$, so that $x$ is locally $F$-analytic by the implicit function theorem for analytic functions (which follows from the inverse function theorem given on page 73 of \cite{SLALG}). This proves the first inclusion. We have $\bt^I_K = \oplus_{i=1}^f \bt^I_F \cdot x_i$, and the reverse inclusion now follows from proposition \ref{basan}, which implies that $(\bt^I_K)^{\fla} = \oplus_{i=1}^f (\bt^I_F)^{\fla} \cdot x_i$, and the case $K=F$.

We now prove (3). Suppose first that $K=F$. Let $r = r_\ell$ and $s \geq r$. Recall that in \S 2 of \cite{FX12}, a map $\psi_q : \bfont^{[r;s]}_F \to \bfont^{[r/q;s/q]}_F$ is constructed, which satisifes $\psi_q(\phi_q(x)) =x$. If $x \in (\btrig{,r}{,F})^{\fpa}$ and $s = r_k$, then item (2) implies that the image of $x$ in $\bt^{[r;s]}_F$ lies in $\bfont^{[r;s]}_{F,m}$ for some $m$. If $t = r_{k'} \geq s$, then there exists likewise $m(t) \geq m$ such that the image of $x$ in $\bt^{[r;t]}_F$ lies in $\bfont^{[r;t]}_{F,m(t)}$ so that $\phi_q^{m(t)}(x) \in \bfont^{[q^{m(t)} r;q^{m(t)} t]}_F$. We have $\phi_q^m(x) = \psi_q^{m(t)-m} \phi_q^{m(t)}(x) \in \bfont^{[q^m r;q^m t]}_F$ so that $x \in \bfont^{[r;t]}_{F,m}$ for all $t\geq s$. The theorem now follows from the fact that $\brig{,r}{,F,m} = \projlim_{s \geq r}  \bfont^{[r;s]}_{F,m}$. If $K \neq F$, then we can write $\btrig{,r}{,K} = \oplus_{i=1}^f \btrig{,r}{,F} \cdot x_i$ as above. The $x_i$'s belong to $\bdag{,r}_K$ and are hence pro-analytic. Item (3) now follows from proposition \ref{baspa}.
\end{proof}

\section{Rings of locally analytic periods}
\label{locanper}

In this section, we prove a number of results about the ring $(\bt^I_K)^{\la}$. In particular, we show that the elements of $(\bt^I_K)^{\la}$ can be written as power series with coefficients in $(\bt^I_K)^{\fla}$. This last result is not used in the rest of the article, but is the underlying motivation for several of our constructions. Let $K_\infty$ be the Lubin-Tate extension of $K$ constructed above. Given $\tau \in \emb$ and $f(Y) =\sum_{k \in \ZZ} a_k Y^k$ with $a_k \in F$, let $f^\tau(Y)=\sum_{k \in \ZZ} \tau(a_k) Y^k$. For $\tau \in \emb$, let $\tilde{n}(\tau)$ be the lift of $n(\tau) \in \ZZ/h\ZZ$ belonging to $\{0,\hdots,h-1\}$. Recall that $E$ is a finite extension of $F$ that contains $F^{\Gal}$ and that if $\tau \in \emb$, then we have $\nabla_\tau \in E \otimes_F \Lie(\Gamma_F)$. The field $E$ is a field of coefficients, so that $G_K$ acts $E$-linearly below.

Let $y_\tau = (\tau \otimes \phi^{\tilde{n}(\tau)}) (u) \in \OO_E \otimes_{\OO_F} \atplus$. We have $g(y_\tau) = [\chilt(g)]^\tau(y_\tau)$ and $\phi_q(y_\tau) = [\pi]^\tau (y_\tau) = \tau (\pi) y_\tau+y_\tau^q$. Let $t_\tau = (\tau \otimes \phi^{\tilde{n}(\tau)})(t_\pi) = \log_{\LT}^\tau(y_\tau)$. Recall that $\weil = \{ (\tau,n) \in \Gal(E/\Qp) \times \ZZ$ such that $n(\tau{\mid_F}) \equiv n \bmod{h}\}$. If $g \in \weil$ and $p^{n(g)-1}(p-1) \in I$, then we have a map $\iota_g : E \otimes_F \bt^I \to E \otimes_F \bdr^+$ given\footnote{recall that $\bt^I$ is a completion of $F \otimes_{\OO_{F_0}} W(\etplus)$ and that $n(g) \equiv n \bmod{h}$, so that $g = \phi^{n(g)}$ on $F_0$} by $x \mapsto (g^{-1} \otimes (g{\mid_F^{-1}} \otimes \phi^{-n(g)})) (x)$.

\begin{lemm}\label{kerim}
If $g \in \weil$ and $p^{n(g)-1}(p-1) \in I$, with $g{\mid_F} = \tau$ and $n(g)-\tilde{n}(\tau) = kh$, then $\ker (\theta \circ \iota_g : E \otimes_F \bt^I \to E \otimes_F \Cp) = Q_k^\tau(y_\tau) \cdot E \otimes_F \bt^I$.
\end{lemm}

\begin{proof}
This follows from the definitions and (\ref{it1}) of lemma \ref{propatrs}.
\end{proof}

We have $\nabla_\tau (y_\tau) = t_\tau \cdot v_\tau$ where $v_\tau = (\partial (T \oplus_{\LT} U)/ \partial U)^\tau(y_\tau,0)$ is a unit (see \S 2.1 of \cite{KR09}). Let $\partial_\tau = t_\tau^{-1} v_\tau^{-1} \nabla_\tau$ so that $\partial_\tau (y_\tau) = 1$. If $\tau,\upsilon \in \emb$, then $\partial_\tau \circ \partial_\upsilon = \partial_\upsilon \circ \partial_\tau$, and $\partial_\tau(y_\upsilon) = 0$ if $\tau \neq \upsilon$.

\begin{lemm}\label{partin}
We have $\partial_\tau((E \otimes_F \btrig{}{,K})^{\pa}) \subset (E \otimes_F \btrig{}{,K})^{\pa}$.
\end{lemm}

\begin{proof}
Take $x \in (\btrig{,r}{,K})^{\pa}$ and take $n=hm+\tilde{n}(\tau)$ with $m$ such that $r_n \geq r$. Let $g \in \weil$ be such that $g{\mid_F}=\tau$ and $n(g)=n$. We have $\theta \circ \iota_g(x) \in E \otimes_F \hat{K}_\infty^{\la}$. Proposition \ref{nabnul} implies that $\nabla_{\Id} = 0$ on $E \otimes_F \hat{K}_\infty^{\la}$ and therefore $\theta \circ \iota_g(\nabla_\tau(x)) = 0$ by lemma \ref{twistder}. By lemma \ref{kerim}, this implies that $\nabla_\tau(x)$ is divisible by $Q_m^\tau(y_\tau)$ in $E \otimes_F \btrig{,r}{,K}$ for all $m$ such that $r_n \geq r$. Since $t_\tau = y_\tau \cdot \prod_{m \geq 1} Q_m^\tau(y_\tau)/ \tau (\pi)$, this implies the lemma (the argument for the division by an infinite product is the same as in lemma 4.6 of \cite{LB2}).
\end{proof}

\begin{lemm}
\label{fondense}
If $x \in \atplus_F$, $I$ is a closed interval not containing $0$, and $n \geq 1$, then there exists $\ell \geq 0$ and $x_n \in \OO_F[\phi_q^{-\ell}(u)]$ such that $x-x_n \in p^n \at_F^I$.
\end{lemm}

\begin{proof}
Let $k \geq 1$ be such that $u^k \in p^n \at_F^I$. By corollary 4.3.4 of \cite{WCN}, the ring $\cup_{m \geq 0} \phi_q^{-m} (\FF_q \dcroc{u})$ is $u$-adically dense in $\etplus_F$. By successive approximations, we find $\ell \geq 0$ and $x_n \in \OO_F[\phi_q^{-\ell}(u)]$ such that $x-x_n \in p^n \atplus_F + u^k \atplus_F$ so that $x-x_n \in p^n \at_F^I$.
\end{proof}

If $n \geq 1$ and $I$ is a closed interval not containing $0$, then lemma \ref{fondense} gives us $y_{\tau,n} \in \OO_E[\phi_q^{-\ell}(u)]$ such that $y_\tau-y_{\tau,n} \in p^n \OO_E \otimes_{\OO_F} \at_F^I$. Note that $y_\tau$ and $y_{\tau,n}$ belong to $(E \otimes_F \bt^I_F)^{\la}$. Let $\emb_0 = \emb \setminus \{ \Id \}$. If $\kbf \in \NN^{\emb_0}$, let $|\kbf|=\sum_{\tau \in \emb_0} k_\tau$ and let $\kbf! = \prod_{\tau \in \emb_0} k_\tau!$ and let $\unbf_\tau$ be the tuple whose entries are $0$ except the $\tau$-th one which is $1$. Let $(\ybf-\ybf_n)^\kbf = \prod_{\tau \in \emb_0} (y_\tau - y_{\tau,n})^{k_\tau}$ and $\partial^\kbf = \prod_{\tau \in \emb_0} \partial^{k_\tau}$. We have
\[ \partial_\tau (\ybf-\ybf_n)^\kbf = 
\begin{cases} 0 & \text{if $k_\tau=0$,} \\ k_\tau (\ybf-\ybf_n)^{\kbf-\unbf_\tau} & \text{if $k_\tau \geq 1$.} \end{cases} \]

By lemma \ref{eventnorm}, there exists $m \geq 1$ such that $y_\tau-y_{\tau,n} \in (E \otimes_F \bt_F^I)^{\Gamma_m\dan}$ and $\|y_\tau-y_{\tau,n}\|_{\Gamma_m} \leq p^{-n}$ for all $\tau \in \emb_0$. Let $\{x_\ibf\}_{\ibf \in \NN^{\emb_0}}$ be a sequence of elements of $(E \otimes_F \bt_K^I)^{\fla,\Gamma_m\dan}$ such that $\| p^{n|\ibf|} x_\ibf \|_{\Gamma_m} \to 0$ as $|\ibf| \to +\infty$. The series $\sum_{\ibf \in \NN^{\emb_0}} x_\ibf (\ybf-\ybf_n)^\ibf$  converges in $(E \otimes_F \bt_K^I)^{\Gamma_m\dan}$.

\begin{theo}
\label{btanqp}
If $x \in (E \otimes_F \bt_K^I)^{\la}$ and $n_0 \geq 0$, then there exists $m,n \geq 1$ and a sequence $\{x_\ibf\}_{\ibf \in \NN^{\emb_0}}$ of $(E \otimes_F \bt_K^I)^{\fla,\Gamma_m\dan}$ such that $\| p^{(n-n_0)|\ibf|} x_\ibf \|_{\Gamma_m} \to 0$ and $x = \sum_{\ibf \in \NN^{\emb}} x_\ibf (\ybf-\ybf_n)^\ibf$. 
\end{theo}

\begin{proof}
Let $m \geq 1$ be such that $x \in (E \otimes_F \bt_K^I)^{\Gamma_m \dan}$. The maps $\partial_\tau : (E \otimes_F \bt_K^I)^{\Gamma_m \dan} \to (E \otimes_F \bt_K^I)^{\Gamma_m \dan}$ are continuous and hence there exists $n \geq 1$ such that $\| \partial^\kbf x \|_{\Gamma_m} \leq p^{(n-n_0-1)|\kbf|} \| x \|_{\Gamma_m}$ for all $\kbf \in \NN^{\emb_0}$. If $\ibf \in \NN^{\emb_0}$, let 
\[ x_\ibf = \frac{1}{\ibf!} \sum_{\kbf \in \NN^{\emb_0}} (-1)^{|\kbf|} \frac{(\ybf-\ybf_n)^\kbf}{\kbf!} \partial^{\kbf+\ibf}(x) . \]
The series above converges in $(E \otimes_F \bt_K^I)^{\Gamma_m\dan}$ to an element $x_\ibf$ such that $\partial_\tau(x_\ibf) = 0$ for all $\tau \in \emb_0$, so that $x_\ibf \in (E \otimes_F \bt_K^I)^{\fla,\Gamma_m\dan}$. In addition, $\|x_\ibf\|_{\Gamma_m} \leq p^{(n-n_0)|\ibf|} \cdot |p^{|\ibf|}/\ibf!|_p \cdot \| x \|_{\Gamma_m}$ so that $\| p^{(n-n_0)|\ibf|} x_\ibf \|_{\Gamma_m} \to 0$, the series $\sum_{\ibf \in \NN^{\emb_0}} x_\ibf (\ybf-\ybf_n)^\ibf$ converges, and its limit is $x$.
\end{proof}

\begin{coro}
\label{parsur}
If $F \neq \Qp$ and $\tau \in \emb$, then $\partial_\tau : (E \otimes_F \bt_K^I)^{\la} \to (E \otimes_F \bt_K^I)^{\la}$ is onto.
\end{coro}

\begin{proof}
Suppose that $\tau \neq \Id$, and write $x = \sum_{\ibf \in \NN^{\emb_0}} x_\ibf (\ybf-\ybf_n)^\ibf$ as in theorem \ref{btanqp}, with $n_0=1$. Since $\partial_\tau (x_\ibf (\ybf-\ybf_n)^\ibf) = x_\ibf i_\tau (\ybf-\ybf_n)^{\ibf-\unbf_\tau}$ if $i_\tau \geq 1$, we have $x=\partial_\tau(z)$ with 
\[ z=\sum_{\ibf \in \NN^{\emb_0}} \frac{x_\ibf}{i_\tau+1} (\ybf-\ybf_n)^{\ibf+\unbf_\tau}. \] 
The series converges because $\|x_\ibf\|_{\Gamma_m} \leq p^{(n-1)|\ibf|} \| x \|_{\Gamma_m}$. If $\tau = \Id$, then one may use the fact that the embeddings play a symmetric role.
\end{proof}

\begin{rema}
Corollary \ref{parsur} does not hold if $F=\Qp$. 
\end{rema}

\section{A multivariable monodromy theorem}
\label{monsec}

Let $K_\infty$ be the Lubin-Tate extension of $K$. In this section, we explain how to descend certain $(E \otimes_F \btrig{}{,K})^{\pa}$-modules to $(E \otimes_F \btrig{}{,K})^{\fpa}$ (theorem \ref{monoconj}). This result is used in \S \ref{krproof} in order to prove theorem \ref{monokr}. Let $\mfont$ be a free $(E \otimes_F \btrig{}{,K})^{\pa}$-module, endowed with a bijective Frobenius map $\phi_q : \mfont \to \mfont$ and with a compatible pro-analytic action of $\Gamma_K$, such that $\nabla_\tau(\mfont) \subset t_\tau \cdot \mfont$ for all $\tau \in \emb_0$. Write $\partial_\tau = v_\tau^{-1} t_\tau^{-1} \nabla_\tau$ so that $\partial_\tau(\mfont) \subset \mfont$ if $\tau \in \emb_0$. Let  
\[ \Sol(\mfont) = \{ x \in \mfont, \text{ such that $\partial_\tau(x) = 0$ for all $\tau \in \emb_0$} \} \] 
so that $\Sol(\mfont)$ is a $(E \otimes_F \btrig{}{,K})^{\fpa}$-module stable under $\Gamma_K$, and such that $\phi_q : \Sol(\mfont) \to \Sol(\mfont)$ is a bijection. Our monodromy theorem is the following result.

\begin{theo}\label{monoconj}
If $\mfont$ is a free $(E \otimes_F \btrig{}{,K})^{\pa}$-module with a bijective Frobenius map $\phi_q$ and a compatible pro-analytic action of $\Gamma_K$, such that $\partial_\tau(\mfont) \subset \mfont$ for all $\tau \in \emb_0$, then $\Sol(\mfont)$ is a free $(E \otimes_F \btrig{}{,K})^{\fpa}$-module, and $\mfont = (E \otimes_F \btrig{}{,K})^{\pa} \otimes_{(E \otimes_F \btrig{}{,K})^{\fpa}} \Sol(\mfont)$.
\end{theo}

\begin{rema}\label{deskf}
The usual monodromy conjecture asks for solutions after possibly performing a finite extension $L/K$ and adjoining a logarithm. In this case, we note the following.
\begin{enumerate}
\item there is no need to perform a finite extension $L/K$ since by an analogue of proposition I.3.2 of \cite{LB10}, a $(E \otimes_F \btrig{}{,L})^{\fpa}$-module with an action of $\Gal(L_\infty/K)$ comes by extension of scalars from a $(E \otimes_F \btrig{}{,K})^{\fpa}$-module with an action of $\Gal(K_\infty/K)$. In the classical case, the coefficients are too small to be able to perform this descent.
\item there is no need to adjoin a log since the maps $\partial_\tau : (E \otimes_F \bt_K^I)^{\la} \to (E \otimes_F \bt_K^I)^{\la}$ are onto.
\end{enumerate}
\end{rema}

\begin{proof}[Proof of theorem \ref{monoconj}]
Let $r \geq 0$ be such that $\mfont$ and all its structures are defined over $(E \otimes_F \btrig{,r}{,K})^{\pa}$, and let $I \subset [r;+\infty[$ be a closed interval, such that $I \cap q I \neq \emptyset$. Let $m_1,\hdots,m_d$ be a basis of $\mfont$ and let $\mfont^I = \oplus_{i=1}^d (E \otimes_F \bt_K^I)^{\la} \cdot m_i$. Let $D_\tau = \Mat(\partial_\tau)$ for $\tau \in \emb_0$. We first prove that $\Sol(\mfont^I)$ is a free $(E \otimes_F \bt_K^I)^{\fla}$-module of rank $d$, such that $\mfont^I = (E \otimes_F \bt_K^I)^{\la} \otimes_{(E \otimes_F \bt_K^I)^{\fla}} \Sol(\mfont^I)$. This amounts to finding a matrix $H \in \GL_d((E \otimes_F \bt_K^I)^{\la})$ such that $\partial_\tau(H) + D_\tau H = 0$ for all $\tau \in \emb_0$. If $\kbf \in \NN^{\emb_0}$, then let $D_\kbf = \Mat(\partial^\kbf)$. 
If $n$ is large enough, then 
\[ H = \sum_{\kbf \in \NN^{\emb_0}} (-1)^{|\kbf|} D_\kbf \frac{(\ybf-\ybf_n)^\kbf}{\kbf!} \]
converges in $\M_d((E \otimes_F \bt_K^I)^{\la})$ to a solution of the equations $\partial_\tau(H) + D_\tau H = 0$ for $\tau \in \emb_0$. If in addition $n \gg 0$, then $\|D_\kbf \cdot (\ybf-\ybf_n)^\kbf/ \kbf! \| < 1$ if $|\kbf| \geq 1$ so that $H \in \GL_d((E \otimes_F \bt_K^I)^{\la})$. The definition of $H$ is inspired by theorem \ref{btanqp}.

This proves that $\Sol(\mfont^I)$ is a free $(E \otimes_F \bt_K^I)^{\fla}$-module of rank $d$ such that $\mfont^I = (E \otimes_F \bt_K^I)^{\la} \otimes_{(E \otimes_F \bt_K^I)^{\fla}} \Sol(\mfont^I)$. It does not seem possible to glue these solutions together for varying $I$ using a Mittag-Leffler process, because the spaces $(E \otimes_F \bt_K^I)^{\fla}$ are LB spaces but not Banach spaces, and the state of the art concerning projective limits of such spaces seems to be insufficient in our case (see the remark preceding theorem 3.2.16 in \cite{JW}). We use instead the Frobenius map to show that we can remain at a finite level, that is work with modules over $E \otimes_F \bfont_{K,n}^I$ for a fixed $n$.

Let $m_1,\hdots,m_d$ be a basis of $\Sol(\mfont^I)$. The Frobenius map $\phi_q$ gives rise to bijections $\phi_q^k :  \Sol(\mfont^I) \to \Sol(\mfont^{q^k I})$ for all $k \geq 0$. Let $J = I \cap q I$ and let $P \in \M_d((E \otimes_F \bt_K^J)^{\la})$ be the matrix of $\phi_q(m_1),\hdots,\phi_q(m_d)$ in the basis $m_1,\hdots,m_d$. We have 
\[ \Sol(\mfont^J) = \Sol((E \otimes_F \bt_K^J)^{\la} \otimes_{(E \otimes_F \bt_K^I)^{\fla}} \Sol(\mfont^I)) = (E \otimes_F \bt_K^J)^{\fla} \otimes_{(E \otimes_F \bt_K^I)^{\fla}} \Sol(\mfont^I), \]
so that $m_1,\hdots,m_d$ is also a basis of $\Sol(\mfont^J)$. Likewise, $\phi_q(m_1),\hdots,\phi_q(m_d)$ is also basis of $\Sol(\mfont^J)$ and hence $P \in \GL_d((E \otimes_F \bt_K^J)^{\fla})$. 

By (2) of theorem \ref{xilocan}, $(E \otimes_F \bt_K^I)^{\fla} = E \otimes_F \bfont_{K,\infty}^I$ so that there exists $n \geq 0$ such that $P \in \GL_d(E \otimes_F \bfont_{K,n}^J)$. For $k \geq 0$, let $I_k = q^k I$ and let $J_k = I_k \cap I_{k+1}$ and $E_k = \oplus_{i=1}^d E \otimes_F \bfont_{K,n}^{I_k} \cdot \phi_q^k(m_i)$. The fact that $P \in \GL_d(E \otimes_F \bfont_{K,n}^J)$ implies that $\phi_q^k(P) \in \GL_d(E \otimes_F \bfont_{K,n}^{J_k})$ and hence 
\[ (E \otimes_F \bfont_{K,n}^{J_k}) \otimes_{E \otimes_F \bfont_{K,n}^{I_k}} E_k = (E \otimes_F \bfont_{K,n}^{J_k}) \otimes_{E \otimes_F \bfont_{K,n}^{I_{k+1}}} E_{k+1} \] 
for all $k \geq 0$. The collection $\{E_k\}_{k \geq 0}$ therefore forms a vector bundle over $E \otimes_F \bfont_{K,n}^{[r;+\infty[}$ for $r = \min(I)$. By theorem 2.8.4 of \cite{KSF} (see also \S 3 of \cite{ST03}), there exist elements $n_1,\hdots,n_d$ of $\cap_{k \geq 0} E_k \subset \mfont$ such that $E_k =  \oplus_{i=1}^d (E \otimes_F \bfont_{K,n}^{I_k}) \cdot n_i$ for all $k \geq 0$. These elements give a basis of $\Sol(\mfont)$ over $(E \otimes_F \btrig{}{,K})^{\fpa}$, which is also a basis of $\mfont$ over $(E \otimes_F \btrig{}{,K})^{\pa}$, and this proves the theorem.
\end{proof}

\section{Lubin-Tate $(\phi,\Gamma)$-modules}
\label{pglt}

We now review the construction of Lubin-Tate $(\phi,\Gamma)$-modules. Let $K_\infty$ be the Lubin-Tate extension of $K$ constructed above. If $K$ is a finite extension of $F$, let $\bfont_K$ be the $p$-adic completion of the field $\bdag{}_K$ defined in \S \ref{ltper}, and let $\afont_K$ denote the ring of integers of $\bfont_K$ for $\vp(\cdot)$. A $(\phi_q,\Gamma_K)$-module over $\bfont_K$ is a finite dimensional $\bfont_K$-vector space $\dfont$, along with a semilinear Frobenius map $\phi_q$ and a commuting continuous and semilinear action of $\Gamma_K$. We say that $\dfont$ is \'etale if $\dfont = \bfont_K \otimes_{\afont_K} \dfont_0$ where $\dfont_0$ is a $(\phi_q,\Gamma_K)$-module over $\afont_K$. Let $\bfont$ be the $p$-adic completion of $\cup_{K/F} \bfont_K$. By specializing the constructions of \cite{F90}, Kisin and Ren prove the following theorem (see theorem 1.6 of \cite{KR09}).

\begin{theo}\label{fontkisren}
The functors $V \mapsto (\bfont \otimes_F V)^{H_K}$ and $\dfont \mapsto (\bfont \otimes_{\bfont_K} \dfont)^{\phi_q=1}$ give rise to mutually inverse equivalences of categories between the category of $F$-linear representations of $G_K$ and the category of \'etale $(\phi_q,\Gamma_K)$-modules over $\bfont_K$.
\end{theo}

We say that a $(\phi_q,\Gamma_K)$-module $\dfont$ is overconvergent if there exists a basis of $\dfont$ in which the matrices of $\phi_q$ and of all $g \in \Gamma_K$ have entries in $\bdag{}_K$.  This basis generates a $\bdag{}_K$-vector space $\ddag{}$ which is canonically attached to $\dfont$. The main result of \cite{CC98} states that if $F=\Qp$,  then every \'etale $(\phi_q,\Gamma_K)$-module over $\bfont_K$ is overconvergent (the proof is given for $\pi=p$, but it is easy to see that it works for any uniformizer). If $F \neq \Qp$, then some simple examples (cf.\ \cite{FX12}) show that this is no longer the case.

We say that an $F$-linear representation $V$ of $G_K$ is $F$-analytic if $\Cp \otimes_F^\tau V$ is the trivial $\Cp$-semilinear representation of $G_K$ for all embeddings $\tau \neq \Id \in \emb$. This definition is the natural generalization of Kisin and Ren's $L$-crystalline representations (\S 3.3.7 of \cite{KR09}). Kisin and Ren go on to show that if $K \subset F_\infty$, and if $V$ is a crystalline $F$-analytic representation of $G_K$, then the $(\phi_q,\Gamma_K)$-module attached to $V$ is overconvergent (see \S 3.3 of \cite{KR09}; they actually prove a stronger result, namely that the $(\phi_q,\Gamma_K)$-module attached to such a $V$ is of finite height). 

More generally, recall that $E$ is a finite extension of $F^{\Gal}$, which is our field of coefficients. Theorem \ref{fontkisren} extends to an equivalence of categories between the category of $E$-linear representations of $G_K$ and the category of \'etale $(\phi_q,\Gamma_K)$-modules over $E \otimes_F \bfont_K$.

\begin{lemm}
\label{fanelin}
If $V$ is an $E$-representation of $G_K$, then the following are equivalent:
\begin{enumerate}
\item $V$ seen as an $F$-representation is $F$-analytic;
\item $\Cp \otimes_E^g V$ is the trivial $\Cp$-semilinear representation of $G_K$ for all $g \in \Gal(E/\Qp)$ such that $g{\mid_F} \neq \Id$.
\end{enumerate}
\end{lemm}

\begin{proof}
If $\tau \in \emb$, then $\Cp \otimes_F^\tau V = \Cp \otimes_F^\tau (E \otimes_E V) = \prod_{g{\mid_F} = \tau} \Cp \otimes_E^g V$.
\end{proof}

If $\drig{}$ is a $(\phi_q,\Gamma_K)$-module over $\brig{}{,K}$, and if $g \in \Gamma_K$ is close enough to $1$, then by standard arguments (see 2.1.2 of \cite{KR09}), the series $\log(g) = \log(1+(g-1))$ gives rise to a differential operator $\nabla_g : \drig{} \to \drig{}$. The map $\mathrm{Lie}\,\Gamma_K \to \End(\drig{})$ arising from $v \mapsto \nabla_{\exp(v)}$ is $\Qp$-linear, and we say that $\drig{}$ is $F$-analytic if this map is $F$-linear (see \S 2.1 of \cite{KR09} and \S 1.3 of \cite{FX12}). This is equivalent to the requirement that the elements of $\drig{}$ be pro-$F$-analytic vectors for the action of $\Gamma_K$.

\begin{theo}\label{mex}
If $V$ is an overconvergent $E$-representation of $G_K$ and if $\drig{}(V) = (E \otimes_F \brig{}{,K}) \otimes_{E \otimes_F \bdag{}_K} \ddag{}(V)$ is $F$-analytic, then $V$ is $F$-analytic.
\end{theo}

\begin{proof}
When $F/\Qp$ is unramified and $K \subset F_\infty$ and $E=F$, this is corollary 4.2 of \cite{PGMLT}. The proof of the general case is similar, and we recall the main arguments below: we prove that if $\drig{}(V)$ is $F$-analytic and if $\tau \in \Gal(E/\Qp)$ is such that $\tau{\mid_F} \neq \Id$, then $V$ is $\Cp$-admissible at $\tau$.

Let $r$ be large enough so that $\dtrig{}(V) = \btrig{}{} \otimes_{\brig{,r}{,K}} \drig{,r}(V)$, let $I$ be a closed subinterval of $[r;+\infty[$, and let $\dfont^I(V) = \bfont^I_K \otimes_{\brig{,r}{,K}} \drig{,r}(V)$. Let $n=hm+\tilde{n}(\tau)$ where $m \geq 0$ is such that $p^{n-1}(p-1) \in I$. Let $g = (\tau,n) \in \weil$. We have the map $\theta \circ \iota_g : E \otimes_F \bt^I \to E \otimes_F \Cp$, giving rise to an isomorphism 
\[ (E \otimes_F \Cp) \otimes^{\theta \circ \iota_g}_{E \otimes_F \bfont^I_K} \dfont^I(V) \to \Cp \otimes^{\tau}_F V. \]

Let $\hat{K}_\infty^{\tau}$ denote the field of locally $\tau{\mid_F}$-analytic vectors of $\hat{K}_\infty$ for the action of $\Gamma_K$. By lemma \ref{twistder}, $\theta \circ \iota_g(E \otimes_F \bfont_K^I ) \subset E \otimes_F \hat{K}_\infty^{\tau}$. Let $\dsen^\tau(V)$ be the $E \otimes_F \hat{K}_\infty^{\tau}$-module
\[ \dsen^\tau(V) = E \otimes_F \hat{K}_\infty^{\tau} \otimes_{\theta \circ \iota_g(E \otimes_F \bfont_K^I )} \theta \circ \iota_g( \dfont^I(V) ). \]
It is free of rank $d = \dim(V)$, its image in $(\Cp \otimes^{\tau}_F V)^{H_K}$ generates $\Cp \otimes^{\tau}_F V$, and its elements are all locally $\tau$-analytic vectors of $(\Cp \otimes^{\tau}_F V)^{H_K}$ because $\drig{}(V)$ is $F$-analytic and $\iota_g \circ \nabla_\tau = \nabla_{\Id} \circ \iota_g$ by lemma \ref{twistder}. If $y \in \dsen^\tau(V)$, then $g(y) = \exp(\tau (\ell (g) ) \cdot \nabla_\tau)(y)$ if $g \in \Gamma_K$ is close to $1$. 

By proposition \ref{fouper}, there exists $x_\tau \in \Cp$ such that $g(x_\tau)-x_\tau = \tau(\ell(g))$ if $g \in G_{F^{\Gal}}$. The element $x_\tau$ belongs to $E \cdot \hat{K}_\infty^{\tau}$ and satisfies $\nabla_\tau(x_\tau)=1$. Take $y \in \dsen^\tau(V)$, and choose $x_{\tau,0} \in E \cdot K_\infty$  such that $|x_\tau-x_{\tau,0}|_p$ is small enough. The series
\[ C(y) = \sum_{k \geq 0} (-1)^k \frac{(x_\tau-x_{\tau,0})^k}{k!} \nabla_\tau^k(y) \]
converges for the topology of $\dsen^\tau(V)$, and a short computation shows that $\nabla_\tau(C(y)) = 0$, so that $C(y) \in (\Cp \otimes^{\tau}_F V)^{G_{E \cdot K_n}}$ for $n = n(y) \gg 0$. In addition, $n(y)=n(\nabla_\tau^k(y))$ for all $k \geq 0$, the series for $C(\nabla_\tau^k(y))$ also converges for the topology of $\dsen^\tau(V)$, and $y= \sum_{k \geq 0} (x_\tau-x_{\tau,0})^k/k! \cdot C(\nabla_\tau^k(y))$.

If $y_1,\hdots,y_d$ is a basis of $\dsen^\tau(V)$, and if $n \geq \max n(y_i)$, then the above computations show that the elements $y_i$ belong to $\hat{K}_\infty^{\tau} \otimes_{K_n} (\Cp \otimes^{\tau}_F V)^{G_{E \cdot K_n}}$, so that $(\Cp \otimes^{\tau}_F V)^{G_{E \cdot K_n}}$ generates $(\Cp \otimes^{\tau}_F V)^{H_K}$. This implies that $V$ is $\Cp$-admissible at the embedding $\tau$.
\end{proof}

In section \ref{krproof}, we prove Theorem C, which says that all $F$-analytic representations of $G_K$ are overconvergent. This was previously known for $F=\Qp$ by the main result of \cite{CC98}, for crystalline representations by \S 3 of \cite{KR09} and for reducible (or even trianguline) $2$-dimensional representations by theorem 0.3 of \cite{FX12}.

\section{The twisted cyclotomic case}
\label{twistcyclo}

Let $K_\infty$ be the Lubin-Tate extension of $K$. Recall that $K_\infty^\eta/K$ is the extension of $K$ attached to $\eta\chicyc$ where $\eta$ is an unramified character of $G_K$. When $\eta=1$, $K_\infty^\eta$ is the cyclotomic extension of $K$ and the Cherbonnier-Colmez theorem (see \cite{CC98}) says that there is an equivalence of categories between \'etale $(\phi,\Gamma)$-modules over $\bdag{}_{K,\cyc}$ and $F$-representations of $G_K$. If $\eta$ is not the trivial character, then there is still such an equivalence of categories, between \'etale $(\phi,\Gamma)$-modules over $\bdag{}_{K,\eta}$ and $F$-representations of $G_K$. This can be seen in at least two ways.
\begin{enumerate}
\item One can redo the whole proof of the Cherbonnier-Colmez theorem for $K_\infty^\eta / K$, and this works because $\Gal(K_\infty^\eta / K)$ is an open subgroup of $\Zp^\times$;
\item One can use the fact that $K_\infty^\eta \cdot \Qpnr = K(\mu_{p^\infty}) \cdot \Qpnr$, apply the classical Cherbonnier-Colmez theorem, and descend from $K_\infty^\eta \cdot \Qpnr$ to $K_\infty^\eta$, which poses no problem since that extension is unramified.
\end{enumerate}
The field $\bdag{}_{K,\eta}$ is a finite extension of the field of power series in a variable $X_\eta$, which is an element of $\OO_{\hat{F}^{\unr}}\dcroc{X_{\cyc}}$, of the form $z_1 X_{\cyc} + \O(\deg \geq 2)$ with $z_1 \in \OO^\times_{\hat{F}^{\unr}}$.

Let $V$ be a $\Qp$-linear representation of $G_K$. By the above generalization of the Cherbonnier-Colmez theorem, $V$ is overconvergent, so that we can attach to $V$ the $\bdag{}_{K,\eta}$-vector space $\ddag{}_\eta(V) = \cup_{r \gg 0} \ddag{,r}_\eta(V)$. Let $\dfont^{[r;s]}_\eta(V)$ and $\ddag{,r}_{\rig,\eta}(V)$ denote the various completions of $\ddag{,r}_\eta(V)$. Let 
\[ \dtilde^{[r;s]}_\eta(V) = (\bt^{[r;s]} \otimes_{\Qp} V)^{G_{K_\infty^\eta}} \quad\text{and}\quad \dtdag{,r}_{\rig,\eta}(V) = (\btrig{,r}{} \otimes_{\Qp} V)^{G_{K_\infty^\eta}}. \]
The Cherbonnier-Colmez theorem implies that $\dtilde^{[r;s]}_\eta(V) =  \bt^{[r;s]}_{K,\eta} \otimes_{\bfont^{[r;s]}_{K,\eta}} \dfont^{[r;s]}_\eta(V)$ and that $\dtdag{,r}_{\rig,\eta}(V) = \btrig{,r}{,K,\eta} \otimes_{\brig{,r}{,K,\eta}} \ddag{,r}_{\rig,\eta}(V)$.

\begin{theo}\label{pgmancy}
We have 
\begin{enumerate}
\item $\dtilde^{[r;s]}_\eta(V)^{\la} = \bfont^{[r;s]}_{K,\eta,\infty} \otimes_{\bfont^{[r;s]}_{K,\eta}} \dfont^{[r;s]}_\eta(V)$;
\item $\dtdag{,r}_{\rig,\eta}(V)^{\pa} =  \brig{,r}{,K,\eta,\infty} \otimes_{\brig{,r}{,K,\eta}} \ddag{,r}_{\rig,\eta}(V)$.
\end{enumerate}
\end{theo}

\begin{proof}
We have $\dtilde^{[r;s]}_\eta(V) =  \bt^{[r;s]}_{K,\eta} \otimes_{\bfont^{[r;s]}_{K,\eta}} \dfont^{[r;s]}_\eta(V)$, and (1) now follows from (2) of theorem \ref{xilocan}, proposition \ref{basan} and from the fact that the elements of $\dfont^{[r;s]}_\eta(V)$ are locally analytic (see \S 2.1 of \cite{KR09}). Likewise, (2) follows from (3) of theorem \ref{xilocan} and proposition \ref{baspa} and from the fact that the elements of $\ddag{,r}_{\rig,\eta}(V)$ are pro-analytic.
\end{proof}

Theorem \ref{pgmancy} shows that in the twisted cyclotomic case, the $(\phi,\Gamma)$-modules over the Robba ring attached to Galois representations can be recovered from the space of locally analytic vectors of certain modules of invariants. It is this observation that inspires the constructions of the next section.

\section{Multivariable $(\phi,\Gamma)$-modules}
\label{pgmult}

We now explain how to construct $(\phi,\Gamma)$-modules over the ring $(\btrig{}{,K})^{\pa}$ for any $p$-adic Lie extension $K_\infty/K$ such that there exists an unramified character $\eta$ of $G_K$ with $K_\infty^\eta \subset K_\infty$ (in particular, the constructions of this section apply to the Lubin-Tate extensions of $K$). Let $\Gamma_K = \Gal(K_\infty/K)$, let $H_K = \Gal(\Qpbar/K_\infty)$, let $V$ be a $p$-adic representation of $G_K$ of dimension $d$, and let 
\[ \dtilde^{[r;s]}_K(V) = (\bt^{[r;s]} \otimes_{\Qp} V)^{H_K} \quad\text{and}\quad \dtdag{,r}_{\rig,K}(V) = (\btrig{,r}{} \otimes_{\Qp} V)^{H_K}. \]
These two spaces are topological representations of $\Gamma_K$.
\begin{theo}\label{lagenok}
We have 
\begin{enumerate}
\item $\dtilde^{[r;s]}_K(V)^{\la} = (\bt^{[r;s]}_K)^{\la} \otimes_{\bfont^{[r;s]}_{K,\eta}} \dfont^{[r;s]}_\eta(V)$;
\item $\dtdag{,r}_{\rig,K}(V)^{\pa} =  (\btrig{,r}{,K})^{\pa} \otimes_{\brig{,r}{,K,\eta}} \ddag{,r}_{\rig,\eta}(V)$.
\end{enumerate}
\end{theo}

\begin{proof}
We have $\bt^{[r;s]} \otimes_{\Qp} V = \bt^{[r;s]} \otimes_{\bfont^{[r;s]}_{K,\eta}} \dfont^{[r;s]}_\eta(V)$, so that $\dtilde^{[r;s]}_K(V) = \bt^{[r;s]}_K \otimes_{\bfont^{[r;s]}_{K,\eta}} \dfont^{[r;s]}_\eta(V)$, and item (1) follows from proposition \ref{basan}. Item (2) is proved similarly.
\end{proof}

Let $\dtdag{}_{\rig,K}(V)^{\pa} = \cup_{r \gg 0} \dtdag{,r}_{\rig,K}(V)^{\pa}$. Theorem \ref{lagenok} implies that $\dtdag{}_{\rig,K}(V)^{\pa}$ is a free $(\btrig{}{,K})^{\pa}$-module of rank $\dim(V)$ which is stable under $\phi_q$ and $\Gamma_K$. We propose this module as a first candidate for a multivariable $(\phi_q,\Gamma_K)$-module in the case in which $\Gamma_K=\Gal(K_\infty/K)$. One can then attempt to construct other multivariable $(\phi,\Gamma)$-modules by descending from $(\btrig{}{,K})^{\pa}$ to certain nicer rings of power series. 

If $K_\infty = K F_\infty$ is a Lubin-Tate extension, and if $\calR(\{Y_\tau\}_{\tau \in \emb})$ denotes the Robba ring in $[F:\Qp]$ variables (see the appendix to \cite{ZS}, applied to the compact $p$-adic Lie group $\OO_F$), theorem \ref{btanqp} shows that $\cup_{n \geq 0} \phi_q^{-n} (\calR(\{Y_\tau\}_{\tau \in \emb}))$ is dense in $(\btrig{}{,K})^{\pa}$ via the map that sends $Y_\tau$ to $y_\tau$. 

\begin{enonce}{Question}
\label{quesdesc}
Which $(\phi,\Gamma)$-modules over $(\btrig{}{,K})^{\pa}$ descend to $\calR(\{Y_\tau\}_{\tau \in \emb})$?
\end{enonce}

For example, if $F$ is unramified over $\Qp$ and $\pi=p$ and $K=F$ and $K_\infty$ is the attached Lubin-Tate extension of $K$ and if $V$ is a crystalline representation of $G_K$, then by theorem A of \cite{PGMLT} one can descend $\dtdag{}_{\rig,K}(V)^{\pa}$ to a reflexive coadmissible module on the ring $\calR^{[0;+\infty[}(Y_0,\hdots,Y_{h-1})$ of functions on the $h$-dimensional open unit disk. 

Note that the cyclotomic element $X= [\varepsilon]-1$ belongs to $(\btrig{}{,K})^{\pa}$, but it is not in the image of $\cup_{n \geq 0} \phi_q^{-n} \calR(\{Y_\tau\}_{\tau \in \emb})$ if $F \neq \Qp$. Therefore, descending to smaller subrings of $(\btrig{}{,K})^{\pa}$ such as $\calR(\{Y_\tau\}_{\tau \in \emb})$ may be quite complicated. 

When $K_\infty/K$ is not a Lubin-Tate extension, it will be interesting to answer the following question.

\begin{enonce}{Question}
\label{questla}
What is the structure of the ring $(\btrig{}{,K})^{\pa}$? 
\end{enonce}

In the appendix to \cite{ZS}, Schneider constructs a (generally noncommutative) Robba ring attached to each compact $p$-adic Lie group $G$. Is the Robba ring attached to $\Gamma_K$ related to $(\btrig{}{,K})^{\pa}$? Finally, we mention that in \cite{KKSM}, Kedlaya discusses some necessary and sufficient conditions for certain elements of $\btrig{}{,K}$ to be locally analytic.

\section{Overconvergence of $F$-analytic representations}
\label{krproof}

We now give the proof of theorems C and D, using the construction of multivariable $(\phi,\Gamma)$-modules and the monodromy theorem. We therefore assume again that $K_\infty$ is the Lubin-Tate extension of $K$.

\begin{theo}\label{monokr}
The Lubin-Tate $(\phi_q,\Gamma_K)$-modules of $F$-analytic representations are overconvergent.
\end{theo}

Let $V$ be an $E$-linear representation of $G_K$ and let $\dtdag{,r}_{\rig,K}(V) = ((E \otimes_F \btrig{,r}{}) \otimes_E V)^{H_K}$. Since $K_\infty$ contains $K_\infty^\eta$, the $E \otimes_F \btrig{,r}{,K}$-module $\dtdag{,r}_{\rig,K}(V)$ is free of rank $d=\dim(V)$ by theorem \ref{lagenok} and there is an isomorphism compatible with $G_K$ and $\phi_q$
\[ (E \otimes_F \btrig{,r}{}) \otimes_{E \otimes_F \btrig{,r}{,K}}  \dtdag{,r}_{\rig,K}(V) = (E \otimes_F \btrig{,r}{}) \otimes_E V. \]

\begin{lemm}\label{monasig}
If $V$ is an $E$-representation of $G_K$ that is $\Cp$-admissible at $\tau \in \emb$, then 
\[ \nabla_\tau  (\dtdag{,r}_{\rig,K}(V))^{\pa} \subset t_\tau \cdot \dtdag{,r}_{\rig,K}(V)^{\pa}. \]
\end{lemm}

\begin{proof}
Take $n=hm+\tilde{n}(\tau)$ with $m$ such that $r_n \geq r$ and let $g \in \weil$ be such that $g{\mid_F}=\tau$ and $n(g)=n$. Let $e_1,\hdots,e_d$ be a basis of $(\Cp \otimes_F^\tau V)^{G_K}$ over $K$, so that it is also a basis of $(\Cp \otimes_F^\tau V)^{H_K}$ over $\hat{K}_\infty$. If $y \in (\Cp \otimes_F^\tau V)^{H_K}$ is $\Qp$-analytic, then we can write $y=\sum_{i=1}^d y_i e_i$ and by lemma \ref{basan}, we have  $y_i \in \hat{K}_\infty^{\la}$. Proposition \ref{nabnul} implies that $\nabla_{\Id} = 0$ on $(\Cp \otimes_F^\tau V)^{H_K}_{\la}$ and therefore that if $x \in \dtdag{,r}_{\rig,K}(V)^{\pa}$, then $\theta \circ \iota_g (\nabla_\tau (x)) = 0$ by lemma \ref{twistder}. Lemma \ref{kerim} implies that if $x \in \dtdag{,r}_{\rig,K}(V)^{\pa}$, then $\nabla_\tau(x)$ is divisible by $Q_m^\tau(y_\tau)$ for all $m$ such that $r_n \geq r$. Since $t_\tau = y_\tau \cdot \prod_{m \geq 1} Q_m^\tau(y_\tau)/\tau (\pi)$, this implies the lemma (the argument for the division by an infinite product is the same as in lemma 4.6 of \cite{LB2}).
\end{proof}

\begin{proof}[Proof of theorem \ref{monokr}]
Let $V$ be an $E$-representation of $G_K$ that is $F$-analytic and let $\mfont  = \dtdag{}_{\rig,K}(V)^{\pa}$. By theorem \ref{lagenok}, $\mfont$ is a free $E \otimes_F (\btrig{}{,K})^{\pa}$-module stable under $\Gamma_K$ and $\phi_q$. Lemma \ref{monasig} implies that $\mfont$ is stable under the differential operators $\{ \partial_\tau \}_{\tau \in \emb \setminus \{ \Id \}}$. By theorem  \ref{monoconj}, $\Sol(\mfont)$ is a free $E \otimes_F (\btrig{}{,K})^{\fpa}$-module of rank $d$ such that there is an isomorphism compatible with $G_K$ and $\phi_q$
\[ (E \otimes_F \btrig{}{}) \otimes_{E \otimes_F (\btrig{}{,K})^{\fpa}} \Sol(\mfont) = (E \otimes_F \btrig{}{}) \otimes_E V. \]
By (3) of theorem \ref{xilocan}, we have $(\btrig{}{,K})^{\fpa} = \brig{}{,K,\infty} = \cup_{n \geq 0} \brig{}{,K,n}$. Since $\Gamma_K$ is topologically of finite type, there exist $n \geq 0$, and a basis $s_1,\hdots,s_d$ of $\Sol(\mfont)$ such that $\Mat(\phi_q) \in \GL_d(E \otimes_F \brig{}{,K,n})$ and $\Mat(g) \in \GL_d(E \otimes_F \brig{}{,K,n})$ for all $g \in \Gamma_K$. If $\drig{} = \oplus_{i=1}^d (E \otimes_F \brig{}{,K}) \cdot \phi_q^n(s_i)$, then $\drig{}$ is a $(\phi_q,\Gamma_K)$-module over $E \otimes_F \brig{}{,K}$ such that $\Sol(\mfont) = (E \otimes_F (\btrig{}{,K})^{\fpa}) \otimes_{E \otimes_F \brig{}{,K}} \drig{}$. 

The module $\drig{}$ is uniquely determined by this condition: if there are two such modules and if $X$ denotes the change of basis matrix and $P_1$, $P_2$ denote the matrices of $\phi_q$, then $X \in \GL_d(E \otimes_F \brig{}{,K,n})$ for $n \gg 0$, and the equation $X=P_2^{-1}\phi(X)P_1$ implies that $X \in \GL_d(E \otimes_F \brig{}{,K})$.

The isomorphism $(E \otimes_F \btrig{}{}) \otimes_{E \otimes_F \brig{}{,K}} \drig{} = (E \otimes_F \btrig{}{}) \otimes_E V$ implies that $\drig{}$ is pure of slope $0$ (see \cite{KSF}). By theorem 6.3.3 of \cite{KSF}, there is an \'etale $(\phi_q,\Gamma_K)$-module $\ddag{}$ over $E \otimes_F \bdag{}_K$ such that $\drig{} = (E \otimes_F \brig{}{,K}) \otimes_{E \otimes_F \bdag{}_K} \ddag{}$. 

Since $\ddag{}$ is \'etale, there exists (by theorem 4.5.7 of \cite{KSF}) an $E$-representation $W$ of $G_K$ such that $(E \otimes_F \btrig{}{}) \otimes_{E \otimes_F \bdag{}_K} \ddag{} = (E \otimes_F \btrig{}{}) \otimes_E W$. Taking $\phi_q$-invariants in $(E \otimes_F \btrig{}{}) \otimes_E W = (E \otimes_F \btrig{}{}) \otimes_E V$ shows that $W=V$. This proves theorem \ref{monokr} for $V$, with $\ddag{}(V) = \ddag{}$.
\end{proof}

\begin{rema}\label{bpair}
The same proof shows that one can attach a (not necessarily \'etale) Lubin-Tate $(\phi_q,\Gamma_K)$-module over $\brig{}{,K}$ to any $F$-analytic $B$-pair (see \cite{LB8}).
\end{rema}

We finish with the statement and proof of theorem D.

\begin{theo}
\label{eqcat}
The functor $V \mapsto \drig{}(V)$ gives rise to an equivalence of categories between the category of $F$-analytic $E$-representations of $G_K$ and the category of \'etale $F$-analytic Lubin-Tate $(\phi_q,\Gamma_K)$-modules over $E \otimes_F \brig{}{,K}$.
\end{theo}

\begin{proof}
If $\drig{}$ is an \'etale $F$-analytic Lubin-Tate $(\phi_q,\Gamma_K)$-module over $E \otimes_F \brig{}{,K}$, then there exists an $E$-representation $V$ of $G_K$ such that $\drig{} = (E \otimes_F \brig{}{,K}) \otimes_{E \otimes_F \bdag{}_K} \ddag{}(V)$. By theorem \ref{mex}, $V$ is $F$-analytic. 
\end{proof}

\subsection*{Acknowledgements} I am grateful to Pierre Colmez for many useful discussions concerning \cite{STLAV} and this paper. In addition, the above constructions are inspired by his observation in the introduction to \cite{CGL2} that ``pour cette \'etude, je disposais d'un certain nombre de points d'appui comme [...] la similitude entre le th\'eor\`eme de Schneider-Teitelbaum sur l'existence de vecteurs localement analytiques et l'existence d'\'el\'ements surconvergents dans n'importe quel $(\phi,\Gamma)$-module \'etale [...]''.

\providecommand{\bysame}{\leavevmode ---\ }
\providecommand{\og}{``}
\providecommand{\fg}{''}
\providecommand{\smfandname}{\&}
\providecommand{\smfedsname}{\'eds.}
\providecommand{\smfedname}{\'ed.}
\providecommand{\smfmastersthesisname}{M\'emoire}
\providecommand{\smfphdthesisname}{Th\`ese}

\end{document}